\batchmode
\makeatletter
\def\input@path{{"D:/ISTA/Notes/Article - Rational singularities for totally negative quiver moment maps/"}}
\makeatother
\documentclass[english]{article}
\usepackage{lmodern}
\usepackage[T1]{fontenc}
\usepackage[latin9]{inputenc}
\usepackage{geometry}
\usepackage{amsmath}
\usepackage{amssymb}
\usepackage{setspace}
\makeatletter
\@ifundefined{date}{}{\date{}}
\usepackage{amsmath}
\usepackage{amsthm}

\usepackage{amsfonts}
\usepackage{mathcomp}
\usepackage{mathrsfs}

\usepackage[all]{xy}
\usepackage{tikz-cd}
\usepackage{youngtab}
\usepackage{hyperref}

\theoremstyle{definition}

\newtheorem{df}{Definition}[subsection]

\theoremstyle{plain}

\newtheorem{thm}[df]{Theorem}
\newtheorem*{thm*}{Theorem}
\newtheorem{cj}[df]{Conjecture}
\newtheorem*{cj*}{Conjecture}
\newtheorem{prop}[df]{Proposition}
\newtheorem*{prop*}{Proposition}
\newtheorem{lem}[df]{Lemma}
\newtheorem*{lem*}{Lemma}
\newtheorem{cor}[df]{Corollary}

\theoremstyle{remark}

\newtheorem{exmp}[df]{Example}
\newtheorem{rmk}[df]{Remark}

\DeclareMathOperator{\Rea}{Re}

\DeclareMathOperator{\Aut}{Aut}
\DeclareMathOperator{\Id}{Id}
\DeclareMathOperator{\Hom}{Hom}
\DeclareMathOperator{\Ext}{Ext}
\DeclareMathOperator{\ext}{ext}

\DeclareMathOperator{\Mat}{Mat}
\DeclareMathOperator{\GL}{GL}

\DeclareMathOperator{\supp}{supp}
\DeclareMathOperator{\rk}{rk}

\DeclareMathOperator{\Rep}{Rep}

\DeclareMathOperator{\Res}{Res}

\DeclareMathOperator{\sm}{sm}
\DeclareMathOperator{\sg}{sg}

\DeclareMathOperator{\Spec}{Spec}

\newcommand\dd{\mathbf{d}}
\newcommand\ee{\mathbf{e}}

\newcommand\KK{\mathbb{K}}
\newcommand\RR{\mathbb{R}}

\newcommand\git{/\!\!/}
\makeatother

\usepackage{babel}
\begin{document}
\title{Rational singularities for moment maps of totally negative quivers}
\author{Tanguy Vernet\thanks{Institute of Science and Technology Austria, Hausel Group}}
\maketitle
\begin{abstract}
We prove that the zero-fiber of the moment map of a totally negative
quiver has rational singularities. Our proof consists in generalizing
dimension bounds on jet spaces of this fiber, which were introduced
by Budur. We also transfer the rational singularities property to
other moduli spaces of objects in 2-Calabi-Yau categories, based on
recent work of Davison.

This has interesting arithmetic applications on quiver moment maps
and moduli spaces of objects in 2-Calabi-Yau categories. First, we
generalize results of Wyss on the asymptotic behaviour of counts of
jets of quiver moment maps over finite fields. Moreover, we interpret
the limit of counts of jets on a given moduli space as its $p$-adic
volume under a canonical measure analogous to the measure built by
Carocci, Orecchia and Wyss on certain moduli spaces of coherent sheaves.
\end{abstract}

\section{Introduction}

Let $Q$ be a quiver and $\dd$ a dimension vector. The moment map
$\mu_{Q,\dd}$ is a central object in the geometric representation
theory of quivers. It is a building block of Lusztig's nilpotent variety
\cite{Lus91,KS97,Lus00} and Nakajima's quiver varieties \cite{Nak94,Nak98},
which were used in geometric realizations of Kac-Moody Lie algebras
and associated quantum groups. A systematic study of geometric properties
of $\mu_{Q,\dd}$ and of quiver varieties was then undertaken by Crawley-Boevey
in the early 2000s \cite{CB01,CB02,CB03a}. Since then, many other
aspects of quiver moment maps have been investigated, such as the
cohomology and singularities of quiver varieties \cite{Hau10,HLRV11,HLRV13a,BS21}
or counts of $\mathbb{F}_{q}$-points of $\mu_{Q,\dd}^{-1}(0)$ \cite{Moz11a,Dav18,Dav23c,BSV20}.
The latter have proved a useful technique to study cohomological Hall
algebras and other counts of quiver representations such as Kac's
polynomials (see \cite{S18} for a survey).

More recently, Wyss observed that the count of jets of $\mu_{Q,\dd}^{-1}(0)$
over finite fields has an interesting asymptotic behaviour \cite{Wys17b}.
Consider the sequence $q^{-n\dim\mu_{Q,\dd}^{-1}(0)}\cdot\sharp\mu_{Q,\dd}^{-1}(0)(\mathbb{F}_{q}[t]/t^{n}),\ n\geq1$.
When $\dd=\underline{1}$, Wyss showed that this sequence converges
when $n$ goes to infinity if, and only if, the graph underlying $Q$
is $2$-connected. Moreover, by computing the local Igusa zeta function
of $\mu_{Q,\underline{1}}^{-1}(0)$, Wyss found an explicit rational
fraction $W_{Q}\in\mathbb{Q}(T)$ in terms of the graphical hyperplane
arrangement associated to $Q$, such that:\[
W_Q(q)=\underset{n\rightarrow +\infty}{\lim}
q^{-n\dim\mu_{Q,\dd}^{-1}(0)}\cdot\sharp\mu_{Q,\dd}^{-1}(0)(\mathbb{F}_{q}[t]/(t^{n})).
\]As discussed in \cite{Wys17b}, $W_{Q}$ enjoys nice numerical properties:
it has a palyndromic numerator $W_{Q}'$; $W_{Q}'$ conjecturally
has non-negative coefficients and is related to the asymptotic analog
of Kac's polynomials over $\mathbb{F}_{q}[t]/(t^{n})$. This is reminiscent
of similar results on Kac's polynomials (i.e. when $n=1$) \cite{CBVB04,Moz11a,Dav23a}
and raises the following questions: can we generalize these asymptotic
counts to higher dimension vectors ($\dd>\underline{1}$)? and can
we find a geometric interpretation of these counts?

\paragraph*{Convergence results for counts of jets}

We can reformulate the first question as follows: for which pairs
$(Q,\dd)$ does this sequence converge, when $n$ goes to infinity?
By work of Musta\c{t}\v{a}, Aizenbud-Avni and Glazer \cite{Mus01,AA18,Gla19},
if $\mu_{Q,\dd}^{-1}(0)$ is a local complete intersection, the sequence
$q^{-n\dim\mu_{Q,\dd}^{-1}(0)}\cdot\sharp\mu_{Q,\dd}^{-1}(0)(\mathbb{F}_{q}[t]/(t^{n})),\ n\geq1$
is bounded precisely when $\mu_{Q,\dd}^{-1}(0)$ has rational singularities.
Using the local Igusa zeta function of $\mu_{Q,\dd}^{-1}(0)$, we
further show that the counts \textit{converge} when $\mu_{Q,\dd}^{-1}(0)$
has rational singularities.

Rational singularities of $\mu_{Q,\dd}$ were brought into focus by
Aizenbud-Avni \cite{AA16} and Budur \cite{Bud21}, when $Q=S_{g}$
is the $g$-loop quiver, in connection with representation growth
of arithmetic groups. They established that $\mu_{S_{g},\dd}^{-1}(0)$
has rational singularities for $g$ large enough. Budur's proof (which
was later generalized to more general moment maps by Herbig, Schwarz
and Seaton \cite{HSS21}) relies on a stratification of $\mu_{Q,\dd}^{-1}(0)$.
Using étale slices, the problem can be reduced to the study of the
deepest stratum of $\mu_{Q,\dd}^{-1}(0)$ for various auxiliary quivers.
Budur introduced an explicit bound on the dimension of jet spaces
over the deepest stratum, by applying earlier results of Crawley-Boevey
\cite{CB03a}. Using this bound, one proves that $\mu_{Q,\dd}^{-1}(0)$
has rational singularities thanks to a criterion by Musta\c{t}\v{a}
\cite{Mus01}.

We identify a large class of pairs $(Q,\dd)$ for which the above
dimension bounds hold in most cases, and which is closed under taking
auxiliary quivers. This is the class of pairs $(Q,\dd)$ where $Q$
is totally negative i.e. the graph underlying $Q$ is complete with
at least two loops at each vertex. Actually, we also require a mild
condition on the dimension vector $\dd$, to ensure that $\mu_{Q,\dd}^{-1}(0)$
is a complete intersection. In that case, we say that the pair $(Q,\dd)$
has property (P) (see Definition \ref{Def/Prop(P)}). Our main result
is the following:

\begin{thm} \label{Thm/MainResIntro}

Let $Q$ be a quiver and $\dd\in\mathbb{N}^{Q_{0}}\setminus\{0\}$
such that $(Q,\dd)$ has property (P). Then $\mu_{Q,\dd}^{-1}(0)$
has rational singularities.

\end{thm}

As mentioned above, the dimension bounds introduced by Budur can be
established for most, but not all pairs $(Q,\dd)$. Actually the bound
fails for some totally negative quivers. This failure occurs for precisely
one pair $(Q,\dd)$ among those considered in \cite{Bud21}, due to
a computational gap in the proof (see Remark \ref{Rmk/BoundFailure}).
For general pairs $(Q,\dd)$ with property (P), we show that the bound
only fails when $\dd=\underline{1}$, in which case we can exploit
Wyss' results on counts of jets to prove that $\mu_{Q,\dd}^{-1}(0)$
has rational singularities regardless. As a corollary, we obtain a
partial answer to our initial question:

\begin{cor} \label{Cor/JetCountMomMap}

Let $Q$ be a quiver and $\dd\in\mathbb{N}^{Q_{0}}\setminus\{0\}$
such that $(Q,\dd)$ has property (P). Then the sequence of rational
numbers\[
q^{-n\dim\mu_{Q,\dd}^{-1}(0)}\cdot\sharp\mu_{Q,\dd}^{-1}(0)(\mathbb{F}_{q}[t]/(t^{n})),\ n\geq1
\]converges when $n$ goes to infinity.

\end{cor}

To the best of our knowledge, classifying pairs $(Q,\dd)$ such that
$\mu_{Q,\dd}^{-1}(0)$ has rational singularities remains an open
problem.

\paragraph*{Asymptotic counts of jets and $p$-adic integrals}

We also interpret the limit of our counts of jets as a $p$-adic volume.
We take a broader perspective and consider moduli stacks of objects
in 2-Calabi-Yau categories. These moduli stacks are locally modelled
on $[\mu_{Q,\dd}^{-1}(0)/\GL(\dd)]$, by work of Davison \cite{Dav21a}.
This makes our total negativity assumption on quivers natural from
a moduli-theoretic point of view: it translates to a homological assumption
on the category we consider, which we call total negativity as well
(see also \cite{DHSM22}). We also require that the locus of simple
objects in the associated moduli space be dense, to ensure that the
local models $\mu_{Q,\dd}^{-1}(0)$ are built from quivers satisfying
property (P).

As before, for counts of jets on moduli stacks to converge, we need
to analyze singularities of those stacks (or their atlases, see Remark
\ref{Rmk/CountJetsStacks}). Our main result implies rational singularities
statements for moduli stacks of totally negative 2-Calabi-Yau categories.
The moduli stacks we consider are quotient stacks of the form $[X/G]$,
where $X$ is a scheme acted on by a linear algebraic group $G$.
The associated moduli spaces are good categorical quotients $X\git G$
in the sense of Geometric Invariant Theory \cite[Ch. 6.]{Dol03}.
It was shown for several moduli stacks of this form that $X\git G$
has rational singularities (more specifically symplectic singularities,
as introduced in \cite{Bea00}). In fewer cases, it was also shown
that $X$ has rational singularities, which implies by a theorem of
Boutot (see \cite[Corollaire]{Bou87}) that $X\git G$ also has rational
singularities. Examples of such moduli spaces include quiver varieties
\cite{BS21,Bud21}, character varieties \cite{Bud21,BZ20}, multiplicative
quiver varieties \cite{ST22,KS23a} and moduli of sheaves on K3 surfaces
\cite{KLS06,AS18,BZ19}. These moduli spaces all parametrize objects
of a 2-Calabi-Yau category, in a precise (differential-graded) sense,
which was formalized by Davison in \cite{Dav21a}. Davison shows,
using a formality argument, that these moduli spaces are étale-locally
modelled on moment maps of quivers. Our theorem goes as follows (see
Theorem \ref{Thm/RatSgTotNeg2CY} for a more precise statement):

\begin{thm} \label{Thm/RatSgTotNeg2CYIntro}

Let $[X/G]$ be a quotient stack parametrizing objects in a totally
negative subcategory of a 2-Calabi-Yau category. Suppose that $X\git G$
contains a dense open subset parametrizing simple objects. Then $X$
is locally complete intersection and has rational singularities.

\end{thm}

Let us give some examples. For representation varieties of fundamental
groups of compact Riemann surfaces, a rational singularities result
was already proved in \cite{Bud21}, as well as for their De Rham
and Dolbeault analogs. We discuss in details rational singularities
for moduli of sheaves on K3 surfaces, as mentioned in the introduction
of \cite{BZ19}. Finally, we prove the following result for representations
spaces of multiplicative preprojective algebras $R(\Lambda^{q}(Q),\dd)$,
thereby dealing with the last example in Davison's list (see \cite[\S1.1.1.]{Dav21a}):

\begin{cor} \label{Cor/MultPreprojIntro}

Let $Q$ be a quiver, $q\in(\mathbb{C}^{\times})^{Q_{0}}$ and $\dd\in\mathbb{N}^{Q_{0}}\setminus\{0\}$
such that $(Q,\dd)$ has property (P). Then $R(\Lambda^{q}(Q),\dd)$
has rational singularities.

\end{cor}

With these results at hand, we can extend our results on counts of
jets to these moduli stacks. If $[X/G]$ is a quotient stack parametrizing
objects in a totally negative 2-Calabi-Yau category, we show that
the sequence $q^{-n\dim X_{\mathbb{Q}}}\cdot\sharp X(\mathbb{F}_{q}[t]/(t^{n})),\ n\geq1$
also converges when $n$ goes to infinity. We interpret the limit
as the $p$-adic volume of an analytic manifold $X^{\natural}$ associated
to $X$, following recent work of Carocci, Orecchia and Wyss \cite{COW24}
(see Theorem \ref{Thm/CanMeas2CYMod} for a more precise statement).
The proof relies on a local description of $X$ as a fiber of a flat
map $\varphi$, whose fibers all have rational singularities, and
a Fubini theorem by Aizenbud and Avni (see \cite{AA16}), applied
to $\varphi$.

\begin{thm} \label{Thm/CanMeas2CYModIntro}

Let $[X/G]$ be a quotient stack parametrizing objects in a totally
negative 2-Calabi-Yau category. Suppose that $X\git G$ contains a
dense open subset parametrizing simple objects.

Then for $p$ large enough, the canonical measure $\mu_{\mathrm{can}}$
on $X^{\natural}$ - introduced in \cite{COW24} - is well-defined.
Moreover, the sequence $q^{-n\dim X_{\mathbb{Q}}}\cdot\sharp X(\mathbb{F}_{q}[t]/t^{n}),\ n\geq1$
converges and its limit is given by:\[
\underset{n\rightarrow +\infty}{\lim}\frac{\sharp X(\mathbb{F}_{q}[t]/(t^{n}))}{q^{n\dim X_{\mathbb{Q}}}}
=
\mu_{\mathrm{can}}(X^{\natural}).
\]

\end{thm}

In \cite{COW24}, the authors recover BPS invariants of some local
surfaces as $p$-adic integrals on a good moduli space of coherent
sheaves. This moduli space is analogous to $X\git G$ in our notation,
instead of $X$. As Kac's polynomials can be interpreted as BPS invariants
of a certain 3-Calabi-Yau category \cite{Moz11a}, we expect a connection
between the counts of jets studied in this paper and BPS invariants
of quivers.

In Sections \ref{Subsect/QuivRep} and \ref{Subsect/MomMap}, we recall
well-known facts on quiver representations and the geometry of moment
maps. In Section \ref{Subsect/MustCrit}, we introduce l.c.i. singularities,
rational singularities and the criterion by Musta\c{t}\v{a} that
we use to prove our main result. In Section \ref{Subsect/EtaleSlices},
we discuss the stratification of $\mu_{Q,\dd}^{-1}(0)$ mentioned
above and how it interacts with rational singularities. In Section
\ref{Subsect/Mod2CY}, we briefly introduce the other moduli spaces
that we deal with in this paper, following \cite{Dav21a}. In Section
\ref{Subsect/RatSgJets}, we discuss arithmetic consequences of rational
singularities, which initially motivated our work in the context of
quiver representations.

Section \ref{Sect/MainRes} is the heart of the paper, where we prove
Theorem \ref{Thm/MainResIntro}. Finally, we apply our main result
in Section \ref{Sect/App2CY} to prove Theorem \ref{Thm/RatSgTotNeg2CYIntro},
Corollary \ref{Cor/JetCountMomMap}, and Theorem \ref{Thm/CanMeas2CYModIntro}.

\paragraph*{Acknowledgements}

I would like to warmly thank Dimitri Wyss for his guidance and supervision
and Nero Budur for helpful discussions and answering all my questions
on his previous works. I would also like to thank Francesca Carocci,
Ben Davison, Lucien Hennecart and Olivier Schiffmann for helpful remarks
and discussions during the writing of this paper. Finally, I would
like to thank the anonymous referees for their careful reading and
suggesting improvements in the exposition. This work was supported
by the Swiss National Science Foundation {[}No. 196960{]}. This project
has also received funding from the European Union\textquoteright s
Horizon 2020 research and innovation programme under the Marie Sk\l odowska-Curie
Grant Agreement No. 101034413.

\paragraph*{Published version}

The version of record of this article, first published in Transformation
Groups, is available online at Publisher\textquoteright s website:
\url{https://doi.org/10.1007/s00031-024-09873-0}

\section{Preliminaries}

\subsection{Quiver representations and their moduli \protect\label{Subsect/QuivRep}}

We collect here elementary notions on quiver representations and their
moduli (see for instance \cite{Rei08a}). A quiver is the datum $Q=(Q_{0},Q_{1},s,t)$
of a set of vertices $Q_{0}$ and a set of arrows $Q_{1}$, along
with source and target maps $s,t:Q_{1}\rightarrow Q_{0}$. Note that
a quiver may have loops and parallel arrows connecting the same vertices.

Let us fix $\KK$ a base field. A representation $V$ of $Q$ is the
datum of $\KK$-vector spaces $(V_{i})_{i\in Q_{0}}$ and $\KK$-linear
maps $(V_{a}:V_{s(a)}\rightarrow V_{t(a)})_{a\in Q_{1}}$. A morphism
between two representations $V$ and $W$ is the datum of $\KK$-linear
maps $(\varphi_{i}:V_{i}\rightarrow W_{i})_{i\in Q_{0}}$ such that,
for all arrows $a\in Q_{1}$, $W_{a}\circ\varphi_{s(a)}=\varphi_{t(a)}\circ V_{a}$.
We will only consider finite-dimensional representations of $Q$ and
define the dimension vector of $V$ as the tuple $\dim(V)=(\dim V_{i})_{i\in Q_{0}}\in\mathbb{N}^{Q_{0}}$.
Finite-dimensional representations of $Q$ over $\KK$ form an abelian
category $\Rep_{\KK}(Q)$.

Given a dimension $\dd\in\mathbb{N}^{Q_{0}}$, we can fix bases of
the vector spaces $V_{i},\ i\in Q_{0}$ and obtain from the linear
maps $V_{a},\ a\in Q_{1}$ a tuple of matrices:\[
(x_a)_{a_\in Q_1}\in R(Q,\dd):=\prod_{a\in Q_1}\Mat(d_{t(a)}\times d_{s(a)},\KK).
\]Conversely, a point $x\in R(Q,\dd)$ yields a $\dd$-dimensional representation
of $Q$. Two points $x,y\in R(Q,\dd)$ correspond to isomorphic representations
if, and only if, they lie in the same orbit of the following action
of $\GL(\dd)=\prod_{i}\GL(d_{i},\KK)$:\[
\begin{array}{cccll}
\GL(\dd) & \times & R(Q,\dd) & \rightarrow & R(Q,\dd) \\
(g_i)_{i\in Q_0} & ; & (x_a)_{a\in Q_1} & \mapsto & (g_{t(a)}x_ag_{s(a)}^{-1})_{a\in Q_1}.
\end{array}
\]The above $\GL(\dd)$-variety is the starting point of quiver moduli.
Assume for simplicity that $\KK$ is algebraically closed, of characteristic
zero. Moduli spaces of quiver representations were constructed by
King \cite{Kin94} using techniques of Geometric Invariant Theory.
The GIT quotient $R(Q,\dd)\git\GL(\dd)$ parametrizes closed orbits
of $R(Q,\dd)$ under the action of $\GL(\dd)$. King further shows
that closed orbits correspond precisely to semisimple representations
of $Q$. Recall that a representation $V$ is called simple if it
contains no non-zero, proper subrepresentation; more generally, $V$
is called semisimple if it decomposes into a direct sum of simple
representations.

We introduce a few more notions which we will use below to define
the class of pairs $(Q,\dd)$ we focus on in this article.

\begin{df}

The Euler form of $Q$ is the following bilinear form:\[
\begin{array}{ccll}
\langle\bullet,\bullet\rangle: & \mathbb{Z}^{Q_0}\times\mathbb{Z}^{Q_0} & \rightarrow & \mathbb{Z} \\
& (\dd,\ee) & \mapsto & \sum_{i\in Q_0}d_ie_i-\sum_{a\in Q_1}d_{s(a)}e_{t(a)}.
\end{array}
\]The symmetrized Euler form is defined by: $(\dd,\ee)=\langle\dd,\ee\rangle+\langle\ee,\dd\rangle$.

\end{df}

\begin{df}

A quiver is called totally negative if $(\dd,\ee)<0$ for all $\dd,\ee\in\mathbb{N}^{Q_{0}}\setminus\{0\}$.

\end{df}

One can check that $Q$ is totally negative if, and only if, (i) $Q$
has at least two loops at each vertex and (ii) any pair of vertices
of $Q$ is joined by at least one arrow. We will use the following
definition in Section \ref{Subsect/MomMap} to describe the geometry
of moment maps of totally negative quivers. We denote by $(\epsilon_{i})_{i\in Q_{0}}$
the canonical basis of $\mathbb{Z}^{Q_{0}}$. 

\begin{df}

The fundamental domain of $Q$ is the following set of dimension vectors:\[
F_{Q}:=\left\{\dd\in\mathbb{N}^{Q_{0}}\setminus\{0\}\  
\left\vert
\begin{array}{l}
\forall i\in Q_0,\ (\dd,\epsilon_i)\leq0 \\
\supp(\dd)\text{ is connected}
\end{array}
\right.
\right\},
\]where $\supp(\dd)$ is the full subquiver of $Q$ with set of vertices
$\{i\in Q_{0}\ \vert\ d_{i}\ne0\}$.

\end{df}

Note that, for a totally negative quiver, $F_{Q}=\mathbb{N}^{Q_{0}}\setminus\{0\}$.

\subsection{Geometry of quiver moment maps \protect\label{Subsect/MomMap}}

In this section, we introduce moment maps associated to quivers and
recall some of their geometric properties, assuming $\KK$ is algebraically
closed of characteristic zero (see also \cite{CB01}). Given a quiver
$Q$, we define its double quiver $\overline{Q}$ by adjoining one
arrow $a^{*}:t(a)\rightarrow s(a)$ to $Q$ for each $a\in Q_{1}$.
Thus $\overline{Q_{1}}=Q_{1}\sqcup Q_{1}^{*}$ and we call $Q^{*}=(Q_{0},Q_{1}^{*},s,t)$
the opposite quiver of $Q$.

The moment map of $Q$ is the following morphism of algebraic varieties:\[
\begin{array}{rcccll}
\mu_{Q,\dd}: & R(Q,\dd) & \times & R(Q^*,\dd) & \rightarrow & \mathfrak{gl}(\dd) \\
 & (x_a)_{a\in Q_1} & ; & (y_a)_{a\in Q_1} & \mapsto & \left(\sum_{\substack{a\in Q_1 \\ t(a)=i}}x_ay_a-\sum_{\substack{a\in Q_1 \\ s(a)=i}}y_ax_a\right)_{i\in Q_0}.
\end{array}
\]We will be interested in the fiber of $\mu_{Q,\dd}$ over $0\in\mathfrak{gl}(\dd)$.
The subscheme $\mu_{Q,\dd}^{-1}(0)$ is invariant under the action
of $\GL(\dd)$ and its orbits correspond to isomorphism classes of
modules over the preprojective algebra $\Pi_{Q}$. We refer to \cite{CB99a}
for a definition of $\Pi_{Q}$, as we will only work with moduli of
$\Pi_{Q}$-modules in this paper, as described above.

Just as for $R(Q,\dd)$, one can form a GIT quotient $\mu_{Q,\dd}^{-1}(0)\git\GL(\dd)$,
whose points parametrize semisimple representations of $\overline{Q}$
whose orbit lie in $\mu_{Q,\dd}^{-1}(0)$, i.e. semisimple $\Pi_{Q}$-modules.
We will be interested in pairs $(Q,\dd)$ such that there exists a
simple $\Pi_{Q}$-module of dimension $\dd$. These were characterized
combinatorially by Crawley-Boevey in \cite{CB01}. Let us now introduce
the class of pairs we focus on in this article.

\begin{df} \label{Def/Prop(P)}

Let $Q$ be a quiver and $\dd\in\mathbb{N}^{Q_{0}}\setminus\{0\}$
a dimension vector. The pair $(Q,\dd)$ has property (P) if:
\begin{enumerate}
\item the quiver $Q$ is totally negative;
\item if $\supp(\dd)$ has two vertices joined by only one edge, then $\dd_{\restriction\supp(\dd)}\ne(1,1)$.
\end{enumerate}
\end{df}

We now show the existence of a simple $\Pi_{Q}$-module for all these
pairs:

\begin{prop} \label{Prop/TotNegSimp}

Suppose that $(Q,\dd)$ has property (P). Then there exists a simple
$\Pi_{Q}$-module of dimension $\dd$.

\end{prop}

\begin{proof}

Since $Q$ is totally negative, $\dd$ lies in the fundamental region
of $Q$. By contradiction, assume that there exist no simple $\Pi_{Q}$-modules
of dimension $\dd$. By \cite[Thm. 8.1.]{CB01}, one of the following
holds:
\begin{enumerate}
\item the subquiver $\supp(\dd)$ is an extended Dynkin quiver with minimal
imaginary root $\delta$ and $\dd=m\delta$ for some $m\geq2$;
\item the subquiver $\supp(\dd)$ splits as a disjoint union $(Q_{0})'\sqcup(Q_{0})''$
such that there is a unique arrow joining $(Q_{0})'$ and $(Q_{0})''$,
say at vertices $i'$ and $i''$, and $d_{i'}=d_{i''}=1$;
\item the subquiver $\supp(\dd)$ splits as a disjoint union $(Q_{0})'\sqcup(Q_{0})''$
such that there is a unique arrow joining $(Q_{0})'$ and $(Q_{0})''$,
say at vertices $i'$ and $i''$, $d_{i'}=1$, the restriction of
$\supp(\dd)$ to $(Q_{0})''$ is an extended Dynkin quiver with minimal
imaginary root $\delta$ and $\dd_{\restriction(Q_{0})''}=m\delta$
for some $m\geq2$.
\end{enumerate}
Cases 1. and 3. cannot happen, as $Q$ has at least two vertices at
each vertex. Since the graph underlying $Q$ is complete, case 2.
can only happen if $\supp(\dd)$ has two vertices joined by a single
edge and $\dd_{\restriction\supp(\dd)}=(1,1)$. This is ruled out
by definition of property (P). Thus we have reached a contradiction
and there exists a simple $\Pi_{Q}$-module of dimension $\dd$. \end{proof}

From Crawley-Boevey's study of the geometric properties of $\mu_{Q,\dd}^{-1}(0)$
\cite[Thm. 1.2. - Cor. 1.4.]{CB01}, we deduce :

\begin{cor} \label{Prop/GeoMomMap}

If $(Q,\dd)$ has property (P), then $\mu_{Q,\dd}^{-1}(0)$ is a reduced,
irreducible complete intersection of dimension $\dd\cdot\dd-1+2\cdot(1-\langle\dd,\dd\rangle)$.
Moreover, $\mu_{Q,\dd}^{-1}(0)\git\GL(\dd)$ has dimension $2\cdot(1-\langle\dd,\dd\rangle)$.

\end{cor}

We prove our main result using dimension estimates on constructible
subsets of $\mu_{Q,\dd}^{-1}(0)$. These rely on older estimates proved
by Crawley-Boevey in \cite[\S 6]{CB03a}. for which we need some additional
notations:

\begin{df}{\cite[\S 1]{CB01}}

Let $M$ be a semisimple $\Pi_{Q}$-module ($Q$ arbitrary), which
decomposes into non-isomorphic direct summands as follows:\[
M\simeq\bigoplus_{i=1}^rM_i^{\oplus e_i}.
\]The semisimple type of $M$ is the multiset $(\dim(M_{i}),e_{i}\ ;\ 1\leq i\leq r)$.

\end{df}

We call $\tau$ simple (resp. strictly semisimple) if it corresponds
to a simple (resp, semisimple, but not simple) representation. Similarly,
we call $x\in\mu_{Q,\dd}^{-1}(0)$ simple (resp. strictly semisimple)
if the corresponding $\Pi_{Q}$-module is. Given a dimension vector
$\dd$, we define $\tau_{\mathrm{min},\dd}:=(\epsilon_{i},d_{i}\ ;\ i\in\supp(\dd))$.
It is the semisimple type of the $\dd$-dimensional $\Pi_{Q}$-module
corresponding to $0\in\mu_{Q,\dd}^{-1}(0)$.

\begin{df}{\cite[\S 6]{CB03a}}

Let $N_{i},\ 1\leq i\leq r$ be a collection of non-isomorphic simple
$\Pi_{Q}$-modules. A $\Pi_{Q}$-module $M$ has top-type $(j_{s},m_{s},\ 1\leq s\leq h)$
if it admits a filtration $0=M_{0}\subsetneq M_{1}\subsetneq\ldots\subsetneq M_{h}=M$,
such that $M_{s}/M_{s-1}\simeq N_{j_{s}}^{\oplus m_{s}}$ and $\hom(M_{s},N_{j_{s}})=m_{s}>0$.

\end{df}

Note that, if $M$ is semisimple with top-type $(j_{s},m_{s},\ 1\leq s\leq h)$
w.r.t. $N_{i},\ 1\leq i\leq r$ and all $N_{i}$ appear as subquotients
of $M$\footnote{This can always be assumed, by possibly forgetting some $N_{i}$.},
then the semisimple type of $M$ is:\[
\tau=\left(\dim(N_i),\sum_{j_s=i}m_s\ ;\ 1\leq i\leq r\right).
\]We can now state Crawley-Boevey's dimension bound:

\begin{prop}{\cite[Lem. 6.2.]{CB03a}} \label{Prop/CBDimBound}

Let $N_{i},\ 1\leq i\leq r$ be a collection of non-isomorphic simple
$\Pi_{Q}$-modules and $(j_{s},m_{s},\ 1\leq s\leq h)$ a top-type.
Define:\[
z_s
=
\left\{
\begin{array}{ll}
0 & \text{ if }\langle\dim(N_{j_s}),\dim(N_{j_s})\rangle=1\text{ or if there is no }t<s\text{ such that } j_t=j_s, \\
m_t & \text{ otherwise, for the largest } t<s\text{ such that } j_t=j_s.
\end{array}
\right.
\]Then the subset of $\mu_{Q,\dd}^{-1}(0)$ corresponding to $\Pi_{Q}$-modules
of top-type $(j_{s},m_{s},\ 1\leq s\leq h)$ is constructible, of
dimension at most:\[
\dd\cdot\dd-1+(1-\langle\dd,\dd\rangle)+\sum_{s=1}^hm_sz_s-\sum_{s=1}^hm_s^2\cdot(1-\langle\dim(N_{j_s}),\dim(N_{j_s})\rangle).
\]\end{prop}

\subsection{Local complete intersection and rational singularities \protect\label{Subsect/MustCrit}}

In this section, we recall results on local complete intersection
and rational singularities, assuming $\KK$ is algebraically closed.
In particular, we state a criterion of Musta\c{t}\v{a} \cite{Mus01}
for a local complete intersection to have rational singularities,
in terms of jet spaces.

We begin with some definitions. Let $X/\KK$ be a scheme of finite
type.

\begin{df}

The scheme $X$ is locally complete intersection (or l.c.i. for short)
if it can be covered by affine open subsets which are complete intersections
in some affine space. In that case, we say that $X$ has l.c.i. singularities.

\end{df}

Having l.c.i. singularities is a property of local rings: $X$ is
locally complete intersection if, and only if, for all $x\in X$,
$\mathcal{O}_{X,x}$ is a complete intersection ring. Moreover, this
is a local property for the smooth topology:

\begin{prop}{\cite[\href{https://stacks.math.columbia.edu/tag/069P}{Tag 069P}]{SP}}
\label{Lem/LciSgDesc}

Let $f:X\rightarrow Y$ be a smooth morphism between schemes of finite
type. If $Y$ has l.c.i. singularities, then $X$ has l.c.i. singularities.
The converse holds if $f$ is surjective.

\end{prop}

We now turn to rational singularities. We further assume that $\KK$
is of characteristic zero.

\begin{df}

The scheme $X$ has rational singularities if for some (hence for
all) resolution of singularities $p:\tilde{X}\rightarrow X$, the
natural morphism $\mathcal{O}_{X}\rightarrow\textbf{R}p_{*}\mathcal{O}_{\tilde{X}}$
is an isomorphism. In other words, the canonical morphism $\mathcal{O}_{X}\rightarrow p_{*}\mathcal{O}_{\tilde{X}}$
is an isomorphism and $\textbf{R}^{i}p_{*}\mathcal{O}_{\tilde{X}}=0$
for all $i>0$. A point $x\in X$ is called a rational singularity
if there exists a Zariski-open neighborhood $U\ni x$ which has rational
singularities.

\end{df}

We will also use a relative notion of rational singularities, introduced
by Aizenbud and Avni \cite{AA16}, in our proof of Theorem \ref{Thm/CanMeas2CYModIntro}.

\begin{df}{\cite[Def. II.]{AA16}}

Let $f:X\rightarrow Y$ be a morphism between smooth, irreducible
varieties over $\KK$. The morphism $f$ is called FRS (flat with
rational singularities) if it is flat and for every $y\in Y(\overline{\KK})$,
the fibre $X\times_{Y}y$ has rational singularities.

\end{df}

If $X$ has rational singularities, then it is normal (hence reduced)
and Cohen-Macaulay - see \cite[II.1.]{Elk78}. While the above definition
might seem quite abstract, one can show directly that having rational
singularities is a local property with respect to smooth morphisms.
This is surely common knowledge for the experts, but we include a
proof for the reader's convenience, as we could not find one in the
literature.

\begin{lem} \label{Lem/RatSgDesc}

Let $f:X\rightarrow Y$ be a smooth morphism between schemes of finite
type. If $Y$ has rational singularities, then $X$ has rational singularities.
The converse holds if $f$ is surjective.

\end{lem}

\begin{proof}

Consider $p_{Y}:\tilde{Y}\rightarrow Y$ a resolution of $Y$ and
$\tilde{X}:=X\times_{Y}\tilde{Y}$. Then, since $f$ is smooth, the
canonical morphism $p_{X}:\tilde{X}\rightarrow X$ is also a resolution
of singularities, so that we get the following cartesian diagram,
where horizontal maps are resolutions of singularities and the vertical
maps are smooth (hence flat):\[
\begin{tikzcd}[ampersand replacement = \&]
\tilde{X}\ar[r, "p_X"]\ar[d, "\tilde{f}"] \& X \ar[d, "f"] \\
\tilde{Y}\ar[r, "p_Y"] \& Y
\end{tikzcd}
\]Suppose that $Y$ has rational singularities. Then flat base change
yields:\[
f^{*}\textbf{R}(p_{Y})_*\mathcal{O}_{\tilde{Y}}\simeq\textbf{R}(p_{X})_*\tilde{f^{*}}\mathcal{O}_{\tilde{Y}}\simeq\textbf{R}(p_{X})_*\mathcal{O}_{\tilde{X}}.
\]Since by assumption $\textbf{R}(p_{Y})_{*}\mathcal{O}_{\tilde{Y}}\simeq\mathcal{O}_{Y}$,
we obtain $\mathcal{O}_{X}\simeq f^{*}\mathcal{O}_{Y}\simeq\textbf{R}(p_{X})_{*}\mathcal{O}_{\tilde{X}}$,
hence $X$ has rational singularities.

Conversely, suppose that $X$ has rational singularities and $f$
is surjective. Then flat base change and the rational singularities
assumption for $X$ yield:\[
f^{*}\textbf{R}(p_{Y})_*\mathcal{O}_{\tilde{Y}}\simeq\textbf{R}(p_{X})_*\tilde{f^{*}}\mathcal{O}_{\tilde{Y}}\simeq\textbf{R}(p_{X})_*\mathcal{O}_{\tilde{X}}\simeq\mathcal{O}_{X}\simeq f^{*}\mathcal{O}_{Y}.
\]Since $f^{*}$ is exact, we obtain taking cohomology sheaves:\[
f^*\textbf{R}^{i}(p_Y)_{*}\mathcal{O}_{\tilde{Y}}\simeq
\left\{
\begin{array}{ll}
f^*\mathcal{O}_{Y} & \text{, if }i=0 \ ; \\
0 & \text{, else.}
\end{array}
\right.
\]By fpqc descent ($f$ is surjective), we finally obtain:\[
\textbf{R}^{i}(p_Y)_{*}\mathcal{O}_{\tilde{Y}}\simeq
\left\{
\begin{array}{ll}
\mathcal{O}_{Y} & \text{, if }i=0 \ ; \\
0 & \text{, else.}
\end{array}
\right.
\]Thus $Y$ has rational singularities, which finishes the proof. \end{proof}

In particular, the above lemma shows that having rational singularities
is an étale-local property. We will use this fact many times in Section
\ref{Subsect/EtaleSlices} in order to transfer rational singularities
from $\mu_{Q,\dd}^{-1}(0)$ to its étale slices.

Finally, we recall the definition of jet spaces. These spaces are
strongly related to singularities of algebraic varieties, via motivic
integration (and as we will see in Section \ref{Subsect/RatSgJets},
p-adic integration). See for example \cite{Mus01,Mus02,ELM04}.

\begin{df}{\cite[\S 3.2.]{CLNS18}}

Let $\KK$ be a field (of any characteristic) and $m\geq0$. Let $X$
be a finite-type $\KK$-scheme. The $m$-th jet scheme of $X$ is
the $\KK$-scheme $X_{m}$ representing the following functor of points:\[
\begin{array}{rcl}
\KK-\mathrm{CAlg} & \rightarrow & \mathrm{Sets} \\
R & \mapsto & X(R[t]/(t^{m+1})).
\end{array}
\]

\end{df}

We now state a criterion by Musta\c{t}\v{a}, which characterizes
rational singularities for locally complete intersection varieties.
This criterion is key to Budur's proof of rational singularities for
moment maps of g-loop quivers.

\begin{prop}{\cite[Thm. 0.1. \& Prop. 1.4.]{Mus01}} \label{Prop/MustCrit}

Let $X$ be a locally complete intersection variety. Let $X_{\sg}$
be its singular locus and $\pi_{m}:X_{m}\rightarrow X$ its $m$-th
jet space. Then $X$ has rational singularities if, and only if, for
all $m\geq1$, $\dim\pi_{m}^{-1}(X_{\sg})<(m+1)\cdot\dim(X)$.

\end{prop}

\subsection{Etale slices \protect\label{Subsect/EtaleSlices}}

In this section, we describe étale slices of $\mu_{Q,\dd}^{-1}(0)$
in terms of semisimple type, following Crawley-Boevey \cite{CB03a}
and Budur \cite{Bud21}. We also explain how one can transfer the
rational singularities property from the étale slices to $\mu_{Q,\dd}^{-1}(0)$.
This technique is at the heart of the inductive reasoning in \cite{Bud21},
which we extend to a larger class of quivers in this paper.

Let $x\in\mu_{Q,\dd}^{-1}(0)$ be a semisimple point of type $\tau=(\dd_{i},e_{i}\ ;\ 1\leq i\leq r)$.
Then its stabilizer satisfies $\GL(\dd)_{x}\simeq\GL(\ee)$. Luna's
étale slice theorem \cite{Lun73} then yields a $\GL(\ee)$-invariant,
locally closed subvariety $S\subseteq\mu_{Q,\dd}^{-1}(0)$ such that
the commutative square below is cartesian, with étale horizontal maps
(the upper horizontal map is given by the action):\[
\begin{tikzcd}[ampersand replacement = \&]
(S\times^{\GL(\dd)_x}\GL(\dd),[x,\Id]) \ar[r]\ar[d] \& (\mu_{Q,\dd}^{-1}(0),x) \ar[d] \\
(S\git\GL(\dd)_x,x) \ar[r] \& (\mu_{Q,\dd}^{-1}(0)\git\GL(\dd),x).
\end{tikzcd}
\]Moreover, by work of Crawley-Boevey \cite{CB03a} and Budur \cite{Bud21},
the étale slice has an étale-local description in terms of a certain
pair $(Q',\ee)$, which satisfies:
\begin{itemize}
\item the set of vertices is $(Q')_{0}=\{1,\ldots,r\}$ and $\ee_{i}:=e_{i}$,
\item the double quiver $\overline{Q'}$ has $2\cdot(1-\langle\dd_{i},\dd_{i}\rangle)$
loops at each vertex $i$ and $-(\dd_{i},\dd_{j})$ arrows from vertex
$i$ to vertex $j$.
\end{itemize}
We call such a quiver an auxiliary quiver attached to semisimple type
$\tau$. Note that $Q'$ is only determined by $\tau$ up to orientation.
In what follows, we will abuse notations and denote by $Q_{\tau}$
any choice of $Q'$, as both $\overline{Q_{\tau}}$ and property (P)
for $Q_{\tau}$ do not depend on orientation.

Recall that for a $G$-variety $X$ with quotient map $q:X\rightarrow X\git G$
($G$ is a reductive group), an open subset $U\subseteq X$ is called
$G$-saturated if $q^{-1}(q(U))=U$. Then the following holds:

\begin{prop}{\cite[\S 4.]{CB03a} \cite[Thm. 2.9.]{Bud21}} \label{Prop/MomMapEtSlice}

There exists a $\GL(\dd)_{x}$-saturated open neighborhood $W\subseteq S$
of $x$ and a $\GL(\ee)$-equivariant\footnote{Using $\GL(\dd)_{x}\simeq\GL(\ee)$.}
morphism $f:(W,x)\rightarrow(\mu_{Q_{\tau},\ee}^{-1}(0),0)$ such
that the commutative diagram below has cartesian squares and étale
horizontal maps:\[
\begin{tikzcd}[ampersand replacement = \&]
(\mu_{Q_{\tau},\ee}^{-1}(0)\times^{\GL(\ee)}\GL(\dd),[0,\Id]) \ar[d] \& (W\times^{\GL(\dd)_x}\GL(\dd),[x,\Id]) \ar[l]\ar[r]\ar[d] \& (\mu_{Q,\dd}^{-1}(0),x) \ar[d] \\
(\mu_{Q_{\tau},\ee}^{-1}(0)\git\GL(\ee),0) \& (W\git\GL(\ee),x) \ar[l]\ar[r] \& (\mu_{Q,\dd}^{-1}(0)\git\GL(\dd),x).
\end{tikzcd}
\]\end{prop}

Therefore analyzing singularities of $\mu_{Q,\dd}^{-1}(0)$ at closed
orbits boils down to analyzing $0\in\mu_{Q_{\tau},\ee}^{-1}(0)$.
The following result of Le Bruyn, Procesi \cite{LBP90} tells us that
there are finitely many semisimple types $\tau$ to consider and that
they come with a partial order (see also \cite[\S2.]{Bud21} for a
detailed exposition). For simplicity, we call $M(Q,\dd):=\mu_{Q,\dd}^{-1}(0)\git\GL(\dd)$.

\begin{prop}{\cite[Thm. 2.2.]{Bud21}} \label{Prop/QuotStrat}

Let $\tau$ be a semisimple type and $M(Q,\dd)_{\tau}$ the subset
of semisimple $\Pi_{Q}$-modules of type $\tau$. Then $M(Q,\dd)_{\tau}$
is locally closed and there are finitely many types $\tau$ such that
$M(Q,\dd)_{\tau}\ne\emptyset$. Moreover:\[
\overline{M(Q,\dd)_{\tau}}=\bigcup_{\tau'\leq\tau}M(Q,\dd)_{\tau'},
\]where $\tau'\leq\tau$ if, and only if, there exist semisimple points
$x',x\in\mu_{Q,\dd}^{-1}(0)$ of types $\tau',\tau$ such that $\GL(\dd)_{x}\subseteq\GL(\dd)_{x'}$.

\end{prop}

Let $q:\mu_{Q,\dd}^{-1}(0)\rightarrow M(Q,\dd)$ be the quotient map.
Then we define $\left(\mu_{Q,\dd}^{-1}(0)\right)_{\tau}:=q^{-1}(M(Q,\dd)_{\tau})$.
We now show how to prove that $\mu_{Q,\dd}^{-1}(0)$ has rational
singularities, also at non-closed orbits.

\begin{prop} \label{Prop/RatSgClosedOrb}

Let $\tau$ be a semisimple type arising from $(Q,\dd)$. Then $\mu_{Q_{\tau},\ee}^{-1}(0)$
has rational singularities if, and only if, $\bigcup_{\tau'\geq\tau}\left(\mu_{Q,\dd}^{-1}(0)\right)_{\tau'}$
has rational singularities.

\end{prop}

\begin{proof}

We first make the following observation: any semisimple point $x\in\left(\mu_{Q,\dd}^{-1}(0)\right)_{\tau}$
has a $\GL(\dd)$-saturated open neighborhood $U_{x}$ contained in
the open subset $\bigcup_{\tau'\geq\tau}\left(\mu_{Q,\dd}^{-1}(0)\right)_{\tau'}=q^{-1}\left(\bigcup_{\tau'\geq\tau}M(Q,\dd)_{\tau'}\right)$.
By possibly shrinking it, one may assume that $U_{x}$ is contained
in the image of the étale morphism $W\times^{\GL(\dd)_{x}}\GL(\dd)\rightarrow\mu_{Q,\dd}^{-1}(0)$
from Proposition \ref{Prop/MomMapEtSlice}. Moreover, for all $\tau'\geq\tau$,
$U_{x}\cap\left(\mu_{Q,\dd}^{-1}(0)\right)_{\tau'}\ne\emptyset$.
Indeed, $q(U_{x})$ is an open neighborhood of $q(x)\in M(Q,\dd)_{\tau}$
and $M(Q,\dd)_{\tau}\subseteq\overline{M(Q,\dd)_{\tau'}}$, so $U_{x}\cap\left(\mu_{Q,\dd}^{-1}(0)\right)_{\tau'}=q^{-1}\left(q(U_{x})\cap M(Q,\dd)_{\tau'}\right)\ne\emptyset$.

Suppose that $\mu_{Q_{\tau},\ee}^{-1}(0)$ has rational singularities.
Then by Proposition \ref{Prop/MomMapEtSlice} and Lemma \ref{Lem/RatSgDesc},
for any semisimple point $x\in\left(\mu_{Q,\dd}^{-1}(0)\right)_{\tau}$,
$U_{x}$ has rational singularities. Since for $\tau'\geq\tau$, $U_{x}\cap\left(\mu_{Q,\dd}^{-1}(0)\right)_{\tau'}\ne\emptyset$,
there exists some semisimple point $x'\in\left(\mu_{Q,\dd}^{-1}(0)\right)_{\tau'}$
whose neighborhood $U_{x'}$ has rational singularities (this may
require shrinking $U_{x'}$). Since all semisimple points in $\left(\mu_{Q,\dd}^{-1}(0)\right)_{\tau'}$
are étale-locally modelled on the same auxiliary quiver $(Q_{\tau'},\ee)$,
we deduce that for \textit{all} $x'\in\left(\mu_{Q,\dd}^{-1}(0)\right)_{\tau'}$,
$U_{x'}$ has rational singularities. Finally, the open subset $\bigcup_{\tau'\geq\tau}\left(\mu_{Q,\dd}^{-1}(0)\right)_{\tau'}\subseteq\mu_{Q,\dd}^{-1}(0)$
is covered by the open neighborhoods $U_{x},\ x\in\bigcup_{\tau'\geq\tau}\left(\mu_{Q,\dd}^{-1}(0)\right)_{\tau'}$,
so it has rational singularities. The converse follows by applying
the same reasoning to $\left(\mu_{Q_{\tau},\ee}^{-1}(0)\right)_{\tau_{\min}}$,
where $\tau_{\min}$ is the semisimple type of $0\in\mu_{Q_{\tau},\ee}^{-1}(0)$.
\end{proof}

\subsection{Local models for moduli of 2-Calabi-Yau categories \protect\label{Subsect/Mod2CY}}

In this section, we gather local models of some moduli stacks of objects
of 2-Calabi-Yau categories, following \cite{Dav21a}. Such local models
were obtained separately in \cite{CB03a,Bud21}, \cite{BGV16,KS23a}
and \cite{AS18,BZ19} and united in the more general framework of
\cite{Dav21a}. In each case, local models are constructed from so-called
Ext-quivers of semisimple (or polystable) objects. The moduli stacks
we consider are also disjoint unions of global quotient stacks $[X_{\alpha}/G_{\alpha}]$,
where $X_{\alpha}$ is a scheme and $G_{\alpha}$ is a reductive group.
This gives us local models for the schemes $X_{\alpha}$, which generalize
the étale slices described above.

Let us first give a definition of 2-Calabi-Yau category which covers
our examples below. For simplicity, we use the definition for triangulated
categories given in \cite[\S2.6.]{Kel08}, although Davison works
with a refined notion of Calabi-Yau structures suited to differential-graded
enhancements of the triangulated categories we consider. As we will
see below, the coarser definition is sufficient for the computations
of Ext-quivers covered in this article. In what follows, $\KK$ denotes
a base field.

\begin{df} \label{Def/CYCat}

Let $d\in\mathbb{Z}$ and $\mathcal{T}$ a $\KK$-linear, Hom-finite,
triangulated category admitting a Serre functor as defined in \cite[\S2.6.]{Kel08}.
We say that $\mathcal{T}$ is $d$-Calabi-Yau if there exists a family
of $\KK$-linear forms $t_{X}:\Hom(X,X[d])\rightarrow\KK,\ X\in\mathcal{T}$
such that, for all $p,q\in\mathbb{Z}$ satisfying $p+q=d$ and for
all $f\in\Hom(X,Y[p])$ and $g\in\Hom(Y,X[q])$, the pairing:\[
\begin{array}{cll}
\Hom(X,Y[p])\times\Hom(Y,X[q]) & \rightarrow & \KK \\
(f,g) & \mapsto & t_X(g[p]\circ f)
\end{array}
\]is non-degenerate and $t_{X}(g[p]\circ f)=(-1)^{pq}\cdot t_{Y}(f[q]\circ g)$.

\end{df}

The triangulated categories considered in the articles are subcategories
of $D^{b}(\mathcal{C})$, where $\mathcal{C}$ is an abelian category
e.g. modules over an algebra or quasi-coherent sheaves over an algebraic
variety. In particular, if $D^{b}(\mathcal{C})$ is $2$-Calabi-Yau,
then for all $X\in\mathcal{C}$, $\Ext_{\mathcal{C}}^{1}(X,X)$ inherits
a symplectic form. Thus $\Ext_{\mathcal{C}}^{1}(X,X)$ must be even-dimensional.

One can also define a notion of Calabi-Yau algebras as in \cite[\S 3.2.]{Gin06}
and \cite[\S 7]{VdB15}. Set $A$ an algebra. Consider the category
$\mathcal{C}$ of right $A$-modules and the triangulated subcategory
$D^{b}(A)$ of $D^{b}(\mathcal{C})$ formed by complexes whose total
cohomology is finite-dimensional. By \cite[\S4.1-2.]{Kel08}, if $A$
is a $d$-Calabi-Yau algebra, then $D^{b}(A)$ is $d$-Calabi-Yau
as a triangulated category. Here is the definition:

\begin{df} \label{Def/CYAlg}

Let $A$ be a $\KK$-algebra and $d\in\mathbb{Z}$. The algebra $A$
is called $d$-Calabi-Yau if:
\begin{enumerate}
\item As an $A-A$-bimodule, $A$ admits a projective resolution of finite
length by finite-dimensional projective $A-A$-bimodules,
\item There is a quasi-isomorphism $\mathrm{RHom}(A,A^{\mathrm{op}}\otimes A)\simeq A[-d]$
of complexes of $A-A$-bimodules.
\end{enumerate}
\end{df}

We now concretely describe the moduli stacks that we deal with in
this paper. Here $\KK$ is algebraically closed, of characteristic
zero.

\subsubsection{$\Pi_{Q}$-modules}

Given a quiver $Q$, we consider the moduli stack $\mathfrak{M}_{\Pi_{Q}}$
of representations of the preprojective algebra $\Pi_{Q}$. It is
the union of the following quotient stacks:\[
\mathfrak{M}_{\Pi_Q}
=\bigsqcup_{\dd\in\mathbb{N}^{Q_0}}\mathfrak{M}_{\Pi_Q}(\dd)
=\bigsqcup_{\dd\in\mathbb{N}^{Q_0}}\left[\mu_{Q,\dd}^{-1}(0)/\GL(\dd)\right].
\]When the underlying graph of $Q$ has no Dynkin diagram of type A,
D, or E among its connected components (in particular, when $Q$ is
totally negative), the algebra $\Pi_{Q}$ is 2-Calabi-Yau \cite[\S 4.2.]{Kel08}.
Moreover:\[
\hom_{\Pi_Q}(M,N)-\ext_{\Pi_Q}^1(M,N)+\ext_{\Pi_Q}^2(M,N)=(\dim(M),\dim(N)).
\]We already saw étale-local models of $\mu_{Q,\dd}^{-1}(0)$ and GIT
quotients $M_{\Pi_{Q}}(\dd):=\mu_{Q,\dd}^{-1}(0)\git\GL(\dd)$ in
Sections \ref{Subsect/MomMap} and \ref{Subsect/EtaleSlices}. Let
us just mention how the auxiliary quivers $(Q_{\tau},\ee)$ are related
to Ext-groups when $\Pi_{Q}$ is $2$-Calabi-Yau. Given a semisimple
$\Pi_{Q}$-module $M=\bigoplus_{i=1}^{r}M_{i}^{\oplus e_{i}}$ of
type $\tau=(\dd_{i},e_{i}\ ;\ 1\leq i\leq r)$, the Ext-quiver of
$M$ is the quiver with set of vertices $\{1,\ldots,r\}$ and $\ext_{\Pi_{Q}}^{1}(M_{i},M_{j})$
arrows from vertex $i$ to vertex $j$. The identity above shows that
the Ext-quiver of $M$ is $\overline{Q_{\tau}}$.

\subsubsection{$\Lambda^{q}(Q)$-modules}

Let $Q$ be a quiver and $q\in(\KK^{\times})^{Q_{0}}$. The multiplicative
preprojective algebra $\Lambda^{q}(Q)$ was introduced by Crawley-Boevey
and Shaw in \cite{CBS06}. We refer to their article for a precise
definition and directly describe moduli of $\Lambda^{q}(Q)$-modules.
Fix a total ordering $<$ of $Q_{1}$. A $\Lambda^{q}(Q)$-module
of dimension $\dd$ is given by a collection of matrices $M_{a},M_{a}^{*},\ a\in Q_{1}$
of size $d_{t(a)}\times d_{s(a)}$ (resp. $d_{s(a)}\times d_{t(a)}$)
such that, for all $a\in Q_{1}$, $I_{d_{t(a)}}+M_{a}M_{a}^{*}$ and
$I_{d_{s(a)}}+M_{a}^{*}M_{a}$ are invertible and:\[
\prod_{\substack{a\in (Q_1,<) \\ t(a)=i}}(1+M_aM_a^*)\times\prod_{\substack{a\in (Q_1,<) \\ s(a)=i}}(1+M_a^*M_a)^{-1}=q_i\cdot I_{d_i}.
\]As in the case of additive preprojective algebras, two collections
of matrices correspond to isomorphic modules if they are conjugated
by an element of $\GL(\dd)$. Therefore, we obtain the moduli stack:\[
\mathfrak{M}_{\Lambda^{q}(Q)}
=\bigsqcup_{\dd\in\mathbb{N}^{Q_0}}\mathfrak{M}_{\Lambda^{q}(Q)}(\dd)
=\bigsqcup_{\dd\in\mathbb{N}^{Q_0}}[R(\Lambda^{q}(Q),\dd)/\GL(\dd)].
\]where $R(\Lambda^{q}(Q),\dd)\subseteq R(\overline{Q},\dd)$ is the
affine, locally closed subvariety defined by the above conditions.
Note that $\Lambda^{q}(Q)$ and $R(\Lambda^{q}(Q),\dd)$ do not depend
(up to isomorphism) on the orientation of $Q$, nor on the choice
of the ordering on $Q_{1}$ - see \cite[Thm. 1.4.]{CBS06}. Moreover,
since $R(\Lambda^{q}(Q),\dd)$ is affine, we obtain a GIT quotient
$M_{\Lambda^{q}(Q)}(\dd)=R(\Lambda^{q}(Q),\dd)\git\GL(\dd)$, which
parametrizes semisimple $\Lambda^{q}(Q)$-modules of dimension $\dd$
similarly to $M_{\Pi_{Q}}(\dd)$.

When $Q$ is connected and contains at least one (not necessarily
oriented) cycle, Kaplan and Schedler proved that $\Lambda^{q}(Q)$
is a $2$-Calabi-Yau algebra (\cite[Thm. 1.2.]{KS23a}). This is the
case, for instance, when $Q$ is totally negative. Moreover, for any
pair of $\Lambda^{q}(Q)$-modules $(M,N)$, the following identity
holds as in the additive case (see \cite[Thm. 1.6.]{CBS06}):\[
\hom_{\Lambda^{q}(Q)}(M,N)-\ext_{\Lambda^{q}(Q)}^1(M,N)+\ext_{\Lambda^{q}(Q)}^2(M,N)=(\dim(M),\dim(N)).
\]Given a semisimple $\Lambda^{q}(Q)$-module $M=\bigoplus_{i=1}^{r}M_{i}^{\oplus e_{i}}$
of type $\tau=(\dd_{i},e_{i}\ ;\ 1\leq i\leq r)$, one constructs
an auxiliary quiver and a dimension vector $(Q_{\tau},\ee)$ as in
the additive case (see Section \ref{Subsect/EtaleSlices}). Likewise,
the identity above shows that the Ext-quiver of $M$ is $\overline{Q_{\tau}}$.

\subsubsection{Semistable coherent sheaves on K3 surfaces}

Let $(S,H)$ be a complex, projective polarized K3 surface. We consider
the moduli stack $\mathfrak{M}_{S,H}$ of Gieseker $H$-semistable,
coherent sheaves on $S$, as described in \cite{HL10}. It is the
union of the following quotient stacks:\[
\mathfrak{M}_{S,H}
=\bigsqcup_{\textbf{v}\in \mathrm{H}^*(S,\mathbb{Z})}\mathfrak{M}_{S,H}(\textbf{v})
=\bigsqcup_{\textbf{v}\in \mathrm{H}^*(S,\mathbb{Z})}[U_{\textbf{v},m_{\textbf{v}}}^{\mathrm{ss}}/\GL(P(m_{\textbf{v}}))].
\]where $\mathfrak{M}_{S,H}(\textbf{v})$ is the substack of semistable
sheaves with Mukai vector $\textbf{v}$. Given a coherent sheaf $\mathcal{F}$
on $S$, its Mukai vector is by definition $\textbf{v}(\mathcal{F}):=(\rk(\mathcal{F}),c_{1}(\mathcal{F}),\frac{1}{2}c_{1}(\mathcal{F})^{2}-c_{2}(\mathcal{F})+\rk(\mathcal{F}))$
and determines the Hilbert polynomial of $\mathcal{F}$. For a given
Mukai vector $\textbf{v}$ and associated Hilbert polynomial $P\in\mathbb{Q}[t]$,
there exists an integer $m_{\textbf{v}}>0$ such that, for any Gieseker
$H$-semistable, coherent sheaf $\mathcal{F}$ on $S$ with Hilbert
polynomial $P$, the sheaf $\mathcal{F}(m_{\textbf{v}})$ is generated
by global sections (see \cite[\S4.3.]{HL10} for details). Such a
sheaf $\mathcal{F}$, together with a choice of a basis of $\Gamma(S,\mathcal{F}(m_{\textbf{v}}))$,
corresponds to a point in $\text{Quot}(\mathcal{O}_{S}(-m_{\textbf{v}})^{\oplus P(m_{\textbf{v}})},P)$.
The locus $U_{\textbf{v},m_{\textbf{v}}}\subseteq\text{Quot}(\mathcal{O}_{S}(-m_{\textbf{v}})^{\oplus P(m_{\textbf{v}})},P)$
of quotient sheaves with Mukai vector $\textbf{v}$ is an open and
closed subset (see \cite[Ch. 10.2.]{Huy16}), which is preserved under
the action of $\GL(P(m_{\textbf{v}}))$ on $\text{Quot}(\mathcal{O}_{S}(-m_{\textbf{v}})^{\oplus P(m_{\textbf{v}})},P)$.
This action admits a linearisation such that the semistable locus
$U_{\textbf{v},m_{\textbf{v}}}^{\mathrm{ss}}\subseteq\text{Quot}(\mathcal{O}_{S}(-m_{\textbf{v}})^{\oplus P(m_{\textbf{v}})},P)$
parametrizes globally generated quotients $\mathcal{O}_{S}(-m_{\textbf{v}})^{\oplus P(m_{\textbf{v}})}\twoheadrightarrow\mathcal{F}$
inducing an isomorphism on global sections. Thus $\mathfrak{M}_{S,H}(\textbf{v})\simeq[U_{\textbf{v},m_{\textbf{v}}}^{\mathrm{ss}}/\GL(P(m_{\textbf{v}}))]$.
Taking the associated GIT quotient, we obtain the moduli space $M(\textbf{v})=M_{S,H}(\textbf{v})$,
which parametrizes $H$-polystable sheaves on $S$. We denote by $M^{s}(\textbf{v})\subseteq M(\textbf{v})$
the open locus of $H$-stable sheaves.

Let us denote Mukai vectors by $\textbf{v}=(r,\textbf{c},a)\in\mathbb{Z}_{\geq0}\oplus\mathrm{NS}(S)\oplus\mathrm{H}^{4}(S,\mathbb{Z})$.
Recall that the Mukai pairing is defined by: $\textbf{v}_{1}\cdot\textbf{v}_{2}=\textbf{c}_{1}\cdot\textbf{c}_{2}-r_{1}a_{2}-r_{2}a_{1}$.
Note that $\mathrm{NS}(S)$ is a lattice (see \cite[Ch. 1.]{Huy16}).
Consider $D^{b}(S)$ the derived category formed by complexes of quasi-coherent
sheaves on $S$ with bounded coherent cohomology. Since $S$ is a
K3 surface, the category $D^{b}(S)$ is 2-Calabi-Yau, by Serre duality.
Moreover, for any pair $\mathcal{F}_{1},\mathcal{F}_{2}$ of semistable
coherent sheaves, the following identity holds \cite[Cor. 6.1.5.]{HL10}:\[
\hom(\mathcal{F}_{1},\mathcal{F}_{2})-\ext^1(\mathcal{F}_{1},\mathcal{F}_{2})+\ext^2(\mathcal{F}_{1},\mathcal{F}_{2})=-\textbf{v}_{1}\cdot\textbf{v}_{2}=-(\textbf{c}_{1}\cdot\textbf{c}_{2}-r_{1}a_{2}-r_{2}a_{1}).
\]Given a polystable sheaf $\mathcal{F}=\bigoplus_{i=1}^{r}\mathcal{F}_{i}^{\oplus e_{i}}$,
with distinct stable summands $\mathcal{F}_{i}$ of Mukai vectors
$\textbf{v}_{i}$, we construct an auxiliary quiver $Q'$ satisfying
the following conditions: $Q'_{0}=\{1,\ldots,r\}$ and the number
of arrows from $i$ to $j$ in $\overline{Q'}$ is:\[
\ext^{1}(\mathcal{F}_{i},\mathcal{F}_{j})
=
\left\{
\begin{array}{ll}
2+\textbf{v}_{i}\cdot\textbf{v}_{i}, & \text{ if } i=j, \\
\textbf{v}_{i}\cdot\textbf{v}_{j}, & \text{ if } i\ne j.
\end{array}
\right.
\]Thus, $\overline{Q'}$ is the Ext-quiver associated to $\mathcal{F}$.
Note that $\ext^{1}(\mathcal{F}_{i},\mathcal{F}_{i})$ is indeed even,
thanks to the $2$-Calabi-Yau property.

\subsubsection*{Local models}

We now describe local models of the above moduli stacks in a unified
manner, following \cite{Dav21a}. Let us first stress common features
of the aforementioned stacks. In what follows, $\mathfrak{M}$ denotes
either $\mathfrak{M}_{\Pi_{Q}}$, $\mathfrak{M}_{\Lambda^{q}(Q)}$
$\mathfrak{M}_{S,H}$ and $\mathcal{A}$ the corresponding category
of objects (the groupoid associated to $\mathcal{A}$ is isomorphic
to $\mathfrak{M}(\KK)$). Then $\mathfrak{M}$ is a union of quotient
stacks:\[
\mathfrak{M}=\bigsqcup_{\alpha}\mathfrak{M}_{\alpha}=\bigsqcup_{\alpha}[X_{\alpha}/G_{\alpha}],
\]where $X_{\alpha}$ is a finite-type $\KK$-scheme and $G_{\alpha}$
is a reductive linear algebraic group. An object $F\in\mathcal{A}$
corresponds to a point $x\in\mathfrak{M}$ and we call $\alpha(F)$
the dimension vector (resp. the Mukai vector) of $F$. If $x\in\mathfrak{M}$
is closed (i.e. the orbit of $x\in X_{\alpha}$ is closed), then $x$
corresponds to a semisimple (or polystable) object $F=\bigoplus_{i=1}^{r}F_{i}^{\oplus e_{i}}\in\mathcal{A}$.
We call $\tau(F)=(\alpha(F_{i}),e_{i}\ ;\ 1\leq i\leq r)$ the type
of $F$ (or $x$). Note that $\Aut(F)\simeq\GL(\ee)$. As explained
above, one can build from $\tau(F)$ an auxiliary quiver $Q_{\tau}$
such that $\overline{Q_{\tau}}$ is the Ext-quiver of $F$.

Another common feature of quotient stacks $\mathfrak{M}_{\alpha}=[X_{\alpha}/G_{\alpha}]\subseteq\mathfrak{M}$
is the existence of good categorical quotients $M_{\alpha}=X_{\alpha}\git G_{\alpha}$
in the sense of Geometric Invariant Theory \cite[Ch. 6.]{Dol03}.
We can stratify $M_{\alpha}$ according to the conjugacy class of
stabilizers $G_{x}\subseteq G_{\alpha},\ x\in X_{\alpha}$ (where
$x$ has a closed orbit). For a reductive subgroup $H\subseteq G_{\alpha}$,
the stratum $(M_{\alpha})_{(H)}$ is locally closed and corresponds
to closed orbits in $X_{\alpha}$ with stabilizer subgroups in the
conjugacy class of $H$. Moreover:\[
\overline{(M_{\alpha})_{(H)}}=\bigcup_{(H')\leq (H)}(M_{\alpha})_{(H')},
\]where $(H')\leq(H)$ if $H$ is conjugated to a subgroup of $H'$.
Note that the type of $x$ determines the conjugacy class of $G_{x}$.

When $X_{\alpha}$ is an affine space with a linear $G_{\alpha}$-action,
properties of this stratification were proved in \cite[Lem. 5.5.]{Sch80}.
If $X_{\alpha}$ is an affine, finite-type $G_{\alpha}$-scheme, then
there exists a closed $G_{\alpha}$-equivariant embedding of $X_{\alpha}$
into a $G_{\alpha}$-representation (see \cite[Prop. 2.3.5.]{Bri17a}),
so we obtain the stratification by restriction. In general, $X_{\alpha}\git G_{\alpha}$
is locally isomorphic to an affine GIT quotient, so we obtain the
stratification by gluing.

\begin{rmk}

In \cite{LBP90}, Le Bruyn and Procesi showed that, for a quiver $Q$
and a dimension vector $\dd$, conjugacy classes of stabilizers of
closed orbits in $R(Q,\dd)$ are in bijection with semisimple types
appearing in $R(Q,\dd)$. However, given a polystable sheaf $\mathcal{F}=\bigoplus_{i=1}^{r}\mathcal{F}_{i}^{\oplus e_{i}}$
on a K3 surface, one can only recover the Hilbert polynomials $P_{\mathcal{F}_{i}},\ 1\leq i\leq r$
from the conjugacy class of its stabilizer and not the Mukai vectors
$\mathbf{v}(\mathcal{F}_{i}),\ 1\leq i\leq r$.

\end{rmk}

The above moduli stacks all arise from a 2-Calabi-Yau category, as
shown in \cite[\S7.]{Dav21a}. Davison works with a dg-enhancement
$\mathcal{T}{}^{\mathrm{dg}}$ of a triangulated category $\mathcal{T}$
containing $\mathcal{A}$. Then $\mathcal{T}{}^{\mathrm{dg}}$ carries
a left 2-Calabi-Yau structure, as defined by Brav and Dyckerhoff \cite{BD19}.
In the examples above, $\mathcal{T}$ is a category of complexes (dg-modules
over the dg-version of $\Pi_{Q}$ or $\Lambda^{q}(Q)$, complexes
of sheaves on a K3 surface) and $\mathcal{A\subset\mathcal{T}}$ is
an abelian subcategory of the category $\mathcal{B\subset\mathcal{T}}$
of complexes concentrated in degree zero. Then $\mathfrak{M}$ sits
as an open substack in the truncation of the derived moduli stack
of objects of $\mathcal{T}{}^{\mathrm{dg}}$, defined by Toën and
Vaquié \cite{TV07}.

Although under certain assumptions, the left 2-Calabi-Yau structure
on $\mathcal{T}{}^{\mathrm{dg}}$ does make $\mathcal{T}$ a 2-Calabi-Yau
category in the sense of Definition \ref{Def/CYCat} (see for instance
\cite[\S 10.1.]{KW21}), in the examples involving quivers, $\mathcal{T}$
may differ from the derived categories of $\Pi_{Q}$ (resp. $\Lambda^{q}(Q)$).
Indeed, in thoses cases, $\mathcal{T}$ is built from the dg-versions
of $\Pi_{Q}$ (resp. $\Lambda^{q}(Q)$). Consequently, in the theorem
below, the Ext-quivers associated to semisimple $\Pi_{Q}$-modules
(resp. $\Lambda^{q}(Q)$-modules) should a priori be computed using
Hom-spaces of $\mathcal{T}$.

However, when $Q$ is totally negative, the dg-version of $\Pi_{Q}$
(resp. $\Lambda^{q}(Q)$) is quasi-isomorphic to $\Pi_{Q}$ itself
(resp. $\Lambda^{q}(Q)$) - see \cite[\S 4.2.]{Kel08} and \cite[Prop. 4.4.]{KS23a}
- so $\mathcal{T}$ is equivalent to the derived category of $\Pi_{Q}$
(resp. $\Lambda^{q}(Q)$). In the case of coherent sheaves on a K3
surface $S$, the dg-category $\mathcal{T}{}^{\mathrm{dg}}$ under
consideration is a dg-enhancement of $D^{b}(S)$ - see \cite[\S 5.2.]{BD19},
\cite[\S 7.2.5.]{Dav21a}. Therefore, in our examples, Ext-quivers
can be computed from Ext-groups in $\mathcal{A}$. For this reason,
we do not give further details and refer to \cite{Dav21a} for a more
thorough discussion of the 2-Calabi-Yau structures at play. We collect
the common features described above in the following definition:

\begin{df} \label{Def/QuotStack2CY}

We call $[X/G]$ a quotient stack coming from a 2-Calabi-Yau category
if:
\begin{enumerate}
\item The stack $[X/G]$ is an open substack of the truncation of the derived
moduli stack of objects of a dg-category $\mathcal{T}{}^{\mathrm{dg}}$
endowed with a left 2-Calabi-Yau structure,
\item Closed points of $[X/G]$ correspond to (a subclass of) semisimple
objects of an abelian, finite-length, $\KK$-linear subcategory $\mathcal{A}\subset\mathcal{T}$;
moreover, if $x\in[X/G](\KK)$ corresponds to the semisimple object
$F=\bigoplus_{i=1}^{r}F_{i}^{\oplus e_{i}}\in\mathcal{A}$, then $(F_{1},\ldots,F_{r})$
is a $\Sigma$-collection in $\mathcal{T}$ (as defined in \cite[\S 1.3.]{Dav21a}),
\item The group $G$ is reductive and $X$ has a good categorical quotient
$X\rightarrow M:=X\git G$.
\end{enumerate}
\end{df}

It follows from the assumptions above that, if $x\in[X/G](\KK)$ corresponds
to the semisimple object $F=\bigoplus_{i=1}^{r}F_{i}^{\oplus e_{i}}\in\mathcal{A}$,
then $\Aut(x)\simeq\GL(\ee)$.

We now describe local models of $X$ and $M$. In the case of $\mathfrak{M}_{S,H}(\textbf{v})$,
the theorem below follows from the formality result in \cite{BZ19}
(see also \cite[\S3-4.]{AS18} and the references therein). In the
case of $\mathfrak{M}_{\Lambda_{Q}^{q}}$, when $\Lambda^{q}(Q)$
is $2$-Calabi-Yau, a local description can be obtained from \cite[Thm. 5.12 - Thm. 5.16.]{KS23a},
using a $G$-equivariant version of Artin's approximation theorem
(see, for instance, \cite{AHR20}).

\begin{thm}{\cite[Thm. 5.11.]{Dav21a}} \label{Thm/LocMod2CY}

Let $[X/G]$ be a quotient stack coming from a 2-Calabi-Yau category
and $x\in[X/G](\KK)$ a closed point corresponding to $F\simeq\bigoplus_{i=1}^{r}F_{i}^{\oplus e_{i}}$.
Let $Q'$ be an auxiliary quiver associated to $F$ (i.e. $\overline{Q'}$
is the Ext-quiver of $F$). Then there exists an affine $\GL(\ee)$-variety
$W$, with a fixed point $w$ and a commutative diagram:\[
\begin{tikzcd}[ampersand replacement = \&]
([\mu_{Q',\ee}^{-1}(0)/\GL(\ee)],0) \ar[d] \& ([W/\GL(\ee)],w) \ar[l]\ar[r]\ar[d] \& ([X/G],x) \ar[d] \\
(M_{\Pi_{Q'},\ee},0) \& (W\git\GL(\ee),w) \ar[l]\ar[r] \& (M,x).
\end{tikzcd}
\]such that the horizontal maps are étale and the squares are cartesian.

\end{thm}

Similarly to section \ref{Subsect/EtaleSlices}, this also gives us
$\GL(\ee)$-equivariant étale morphisms $(W,w)\rightarrow(\mu_{Q',\ee}^{-1}(0),0)$
and $(W,w)\rightarrow(X,x)$ which induce:\[
\begin{tikzcd}[ampersand replacement = \&]
(\mu_{Q',\ee}^{-1}(0)\times^{\GL(\ee)}G,[0,\Id]) \ar[d] \& (W\times^{\GL(\ee)}G,[w,\Id]) \ar[l]\ar[r]\ar[d] \& (X,x) \ar[d] \\
(M_{\Pi_{Q'},\ee},0) \& (W\git\GL(\ee),w) \ar[l]\ar[r] \& (M,x).
\end{tikzcd}
\]where horizontal maps are étale and squares are cartesian.

\subsection{Rational singularities and counts of jets \protect\label{Subsect/RatSgJets}}

In this section, we shift gears and introduce arithmetic characterizations
of rational singularities, based on works of Aizenbud-Avni \cite{AA18}
and Glazer \cite{Gla19}. In order to count $\mathbb{F}_{q}$-rational
points of jet spaces, we consider a finite-type scheme $X/\mathbb{Z}$.
When $X$ is locally complete intersection, Aizenbud-Avni and Glazer
showed that the sequences $q^{-n\dim X_{\mathbb{Q}}}\cdot\sharp X(\mathbb{F}_{q}[t]/(t^{n})),\ n\geq1$
are bounded for all finite fields precisely when $X_{\bar{\mathbb{Q}}}$
has rational singularities (see \cite[Thm. 4.1.]{Gla19} for the detailed
assumptions). There is more: by studying the poles of the local Igusa
zeta function associated to an affine scheme, Wyss showed that the
sequence $q^{-n\dim X_{\mathbb{Q}}}\cdot\sharp X(\mathbb{F}_{q}[t]/(t^{n})),\ n\geq1$,
when bounded, has a limit when $n$ goes to infinity \cite[Lem. 4.7.]{Wys17b}.
We will identify this limit with a p-adic integral on an analytic
manifold associated to $X$ in Section \ref{Subsect/ArithmCor}.

Suppose $X=V(f_{1},\ldots,f_{m})\subseteq\mathbb{A}_{\mathbb{Z}}^{r}$.
The local Igusa zeta function $Z_{f}(s)$ is defined as a parametric
$p$-adic integral over $\mathbb{A}^{r}$ (see \cite{Igu00} for an
introduction). Consider $F$ a finite extension of $\mathbb{Q}_{p}$,
with valuation ring $\mathcal{O}\subset F$, maximal ideal $\mathfrak{m}\subseteq\mathcal{O}$
and residue field $\mathcal{O}/\mathfrak{m}\simeq\mathbb{F}_{q}$.
By definition:\[
Z_f(s):=\int_{\mathcal{O}^r}\Vert(f_1(x),\ldots,f_m(x))\Vert^s dx,
\]where $s\in\mathbb{C}$ and for $(y_{1},\ldots,y_{m})\in F^{m}$,
$\Vert(y_{1},\ldots,y_{m})\Vert:=\underset{1\leq j\leq m}{\max}\vert y_{j}\vert$
and $\vert\bullet\vert$ is the usual non-archimedean norm on $F$.

The local Igusa zeta function encodes the sequence $\sharp X(\mathcal{O}/\mathfrak{m}^{n}),\ n\geq1$
into a generating series:\[
\sum_{n\geq1}\sharp X(\mathcal{O}/\mathfrak{m}^{n})\cdot q^{-n(r+s)}=\frac{Z_f(s)-1}{1-q^s}.
\]Of course, in this paper, we are interested in the counts $\sharp X(\mathbb{F}_{q}[t]/(t^{n}))$
instead of $\sharp X(\mathcal{O}/\mathfrak{m}^{n})$. The following
result by Aizenbud and Avni\footnote{This result follows from transfer results for p-adic integrals, see
\cite{CL05,CL10}. Therefore, it holds for any choice of $F$, as
opposed to \cite{AA18}, where the authors only work with unramified
extensions of $\mathbb{Q}_{p}$.} tells us that, up to working in large enough characteristic, the
two counts actually coincide.

\begin{prop}{\cite[Prop. 3.0.2.]{AA18}} \label{Prop/TsfPrinc}

Let $X$ be a $\mathbb{Z}$-scheme of finite type. There is a finite
set of primes $S$ such that, for any $p\notin S$, for any $q$ power
of $p$ and for any $n\geq1$, $\sharp X(\mathbb{F}_{q}[t]/(t^{n}))=\sharp X(\mathcal{O}/\mathfrak{m}^{n})$.

\end{prop}

The local Igusa zeta function also enjoys an explicit formula as a
rational function in $T=q^{-s}$. This was proved by Denef when $r=1$
\cite{Den87} and by Veys and Zuniga-Galindo \cite{VZG08} in the
general case. The formula relies on a principalization $\pi:W\rightarrow\mathbb{A}_{\mathbb{Q}}^{r}$
of the ideal $I=(f_{1},\ldots,f_{m})\subseteq\mathbb{Q}[x_{1},\ldots,x_{r}]$
- see \cite{VZG08} for details. Then the ideal $\pi^{*}I$ defines
a divisor $\sum_{i\in T}N_{i}E_{i}$ with simple normal crossings,
where $E_{i},\ i\in T$ are the irreducible components of $\pi^{-1}(X_{\mathbb{Q}})$.
We will also need the following numerical data: $\mathrm{div}(\pi^{*}(dx_{1}\wedge\ldots\wedge dx_{r}))=\sum_{i\in T}(\nu_{i}-1)E_{i}$.
The quantity $\mathrm{lct}(\mathbb{A}_{\mathbb{Q}}^{r},X_{\mathbb{Q}})=\min_{i\in T}\left\{ \frac{\nu_{i}}{N_{i}}\right\} $
is called the log-canonical threshold of the pair $(\mathbb{A}_{\mathbb{Q}}^{r},X_{\mathbb{Q}})$.
This is an important invariant in singularity theory and it is related
to jet spaces of $X_{\mathbb{Q}}$ (see for instance \cite{Mus02}).
Finally, we refer to \cite[\S 2.]{Den87} for the notion of good reduction
of $\pi_{F}$ modulo $\mathfrak{m}$ and denote by $\overline{\pi}$,
$\overline{E_{i}}$ and $\overline{X}$ the reductions of $\pi_{F}$,
$(E_{i})_{F}$ and $X_{F}$. These are schemes and morphisms defined
over $\mathbb{F}_{q}$.

\begin{prop}{\cite[Thm. 2.4. - Thm. 3.1.]{Den87} \cite[Thm. 2.10.]{VZG08}}
\label{Prop/DenefFml}

There is a finite set of primes $S$ such that for $p\notin S$, $q$
a prime power of $p$ and $F$ any finite extension of $\mathbb{Q}_{p}$
with residue field $\mathbb{F}_{q}$, $\pi_{F}$ has good reduction
modulo $\mathfrak{m}$. In that case,\[
Z_{f}(s)=q^{-r}\cdot\sum_{J\subseteq T}c_J\prod_{j\in J}\frac{(q-1)q^{-N_js-\nu_j}}{1-q^{-N_js-\nu_j}},
\] where $c_{J}=\sharp\{a\in\overline{X}(\mathbb{F}_{q})\ \vert\ a\in\overline{E_{j}}(\mathbb{F}_{q})\Leftrightarrow j\in J\}$.

\end{prop}

In \cite{Wys17b}, Wyss interprets the limit when $n$ goes to infinity
of $q^{-n\dim X_{\mathbb{Q}}}\cdot\sharp X(\mathbb{F}_{q}[t]/(t^{n}))$
as a residue of $Z_{f}(s)$, seen as a rational fraction in $T$.
Summing up results of \cite{Wys17b,AA18,Gla19}, we obtain:

\begin{prop} \label{Prop/JetCountLim}

Set $X=V(f_{1},\ldots,f_{m})\subseteq\mathbb{A}_{\mathbb{Z}}^{r}$
as above and assume that $X_{\mathbb{\bar{\mathbb{Q}}}}$ is an equidimensional
local complete intersection, of codimension $c$ in $\mathbb{A}_{\bar{\mathbb{Q}}}^{r}$.
Then the following are equivalent:
\begin{enumerate}
\item the scheme $X_{\bar{\mathbb{Q}}}$ has rational singularities,
\item for almost all primes $p$, $q^{-n\dim X_{\mathbb{Q}}}\cdot\sharp X(\mathbb{F}_{q}[t]/(t^{n})),\ n\geq1$
converges for any finite field $\mathbb{F}_{q}$ of characteristic
$p$.
\end{enumerate}
If 1. and 2. hold, there is a finite set of primes $S$ such that
for $p\notin S$, $q$ a \textit{large enough} prime power of $p$
and $F$ any finite extension of $\mathbb{Q}_{p}$ with residue field
$\mathbb{F}_{q}$, $Z_{f}(s)$ has its poles of largest real part
at $\Rea(s)=-c$ and:\[
\underset{n\rightarrow +\infty}{\lim}\frac{\sharp X(\mathbb{F}_{q}[t]/(t^{n}))}{q^{n\dim X_{\mathbb{Q}}}}=\frac{-1}{q^c-1}\cdot\Res_{T=q^c}Z_f.
\]\end{prop}

\begin{proof}

The implication 2. $\Rightarrow$ 1. is a consequence of the implication
the implication (v) $\Rightarrow$ (iii) from \cite[Thm. 4.1.]{Gla19}.
Glazer's result involves $\sharp X(\mathcal{O}/\mathfrak{m}^{n})$
instead of $\sharp X(\mathbb{F}_{q}[t]/(t^{n}))$, where $\mathcal{O}$
is the valuation ring of the unramified extension of $\mathbb{Q}_{p}$
with residue field $\mathcal{O}/\mathfrak{m}\simeq\mathbb{F}_{q}$.
However, for a given $X$, these counts are equal for almost all characteristics
by Proposition \ref{Prop/TsfPrinc}.

Conversely, suppose 1. Then by implication (iii) $\Rightarrow$ (iv')
from \cite[Thm. 4.1.]{Gla19}, for almost all primes $p$, $q^{-n\dim X_{\mathbb{Q}}}\cdot\sharp X(\mathbb{F}_{q}[t]/(t^{n})),\ n\geq1$
is bounded for any finite field $\mathbb{F}_{q}$ of characteristic
$p$. Note that we used again Proposition \ref{Prop/TsfPrinc} in
order to replace $\sharp X(\mathcal{O}/\mathfrak{m}^{n})$ by $\sharp X(\mathbb{F}_{q}[t]/(t^{n}))$
in the conclusion of Glazer's theorem.

Moreover, since $X_{\bar{\mathbb{Q}}}$ has rational singularities,
$Z_{f}$ has its poles of largest real part at $\text{Re}(s)=-\mathrm{lct}(\mathbb{A}_{\mathbb{Q}}^{r},X_{\mathbb{Q}})=-c$.
The first equality is due to \cite[Thm. 2.7.]{VZG08}, while the second
follows from \cite[Prop. 1.4.]{Mus01} and \cite[Cor. 0.2.]{Mus02}.
Note that the proof of \cite[Thm. 2.7.]{VZG08} requires the existence
of a generic $F$-point on a divisor $(E_{i})_{F}$ satisfying $\mathrm{lct}(\mathbb{A}_{\mathbb{Q}}^{r},X_{\mathbb{Q}})=\frac{\nu_{i}}{N_{i}}$
(see also \cite[Rmk. 2.8.2.]{VZG08}). We can assume this is true
by taking a finite extension of $F$, which is why we require that
$q$ be a large enough power of $p$. Then we can apply \cite[Lem. 4.7.]{Wys17b}
and conclude that $q^{-n\dim X_{\mathbb{Q}}}\cdot\sharp X(\mathbb{F}_{q}[t]/(t^{n})),\ n\geq1$
converges and that its limit is given by the residue formula above.
\end{proof}

\begin{rmk}

The proof by Aizenbud and Avni \cite[Thm. 3.0.3.]{AA18} of the implication
1. $\Rightarrow$ 2. in Proposition \ref{Prop/JetCountLim} relies
on the Lang-Weil estimates, Denef's explicit formula for $Z_{f}(s)$
and a characterisation by Musta\c{t}\u{a} of top-dimensional irreducible
components of jet schemes of $X_{\bar{\mathbb{Q}}}$ in terms of the
quantities $(N_{i},\nu_{i})$ \cite[Thm. 3.2.]{Mus01} (when $X_{\bar{\mathbb{Q}}}$
is irreducible). None of these require that $X_{\mathbb{\bar{\mathbb{Q}}}}$
be l.c.i. For l.c.i. varieties, the irreducibility of jet schemes
is implied by the fact that $X$ has rational singularities \cite[Thm. 3.3.]{Mus01}.
But irreducibility of jet schemes also follows directly from the assumption
of Proposition \ref{Prop/MustCrit} on the dimension of jet schemes
over the singular locus of $X_{\bar{\mathbb{Q}}}$. Hence, even if
$X_{\mathbb{\bar{\mathbb{Q}}}}$ is not l.c.i., if $\dim\pi_{m}^{-1}(X_{\sg})<(m+1)\cdot\dim(X)$
for all $m\geq1$, then the sequence $q^{-n\dim X_{\mathbb{Q}}}\cdot\sharp X(\mathbb{F}_{q}[t]/(t^{n})),\ n\geq1$
converges.

Instead, our motivation for working with locally complete intersection
varieties has to do with the fact that the above properties of jet
schemes become étale-local under this assumption (by Proposition \ref{Lem/RatSgDesc},
since they correspond to rational singularities). Moreover, locally
complete intersection varieties are the appropriate setting to compute
counts of jets as p-adic integrals (see the proof of Theorem \ref{Thm/CanMeas2CYModIntro}
below).

\end{rmk}

As a consequence of the counts of jets carried out in \cite{Wys17b},
we obtain rational singularities for certain fibers of quiver moment
maps. This result is necessary to complete the proof of Theorem \ref{Thm/MainResIntro},
as explained in Remark \ref{Rmk/BoundFailure}. When $X=\mu_{Q,\dd}^{-1}(0)$
and $\dd=\underline{1}$, Wyss gave a criterion for the existence
of this limit and an explicit formula in terms of the graphical hyperplane
arrangement associated to $Q$ \cite[Cor. 4.27.]{Wys17b}. Our initial
motivation for this work was to extend this result to higher dimension
vectors.

\begin{prop}{\cite[Cor. 4.27.]{Wys17b}} \label{Prop/ToricFormulaB}

Let $Q$ be a quiver and $\dd=\underline{1}$. Let $p$ be a large
enough prime and $\mathbb{F}_{q}$ be any finite field of characteristic
$p$. Then $q^{-n\dim\mu_{Q,\dd}^{-1}(0)}\cdot\sharp\mu_{Q,\dd}^{-1}(0)(\mathbb{F}_{q}[t]/(t^{n})),\ n\geq1$
converges if, and only if, the underlying graph of $Q$ is 2-connected
i.e. removing one edge does not disconnect the graph. In that case,
there exists an explicit rational fraction $W\in\mathbb{Q}(T)$ depending
only on the underlying graph of $Q$ such that:\[
\underset{n\rightarrow+\infty}{\lim}\frac{\sharp\mu_{Q,\dd}^{-1}(0)(\mathbb{F}_{q}[t]/(t^{n}))}{q^{n\dim\mu_{Q,\dd}^{-1}(0)}}=W(q).
\]\end{prop}

When $\dd=1$, there exists a simple $\Pi_{Q}$-module of dimension
$\dd$ exactly when $Q$ is $2$-connected. Indeed, for a connected
quiver $Q$, $b(Q):=1-\sharp Q_{0}+\sharp Q_{1}=1-\langle\dd,\dd\rangle$
and $Q$ is $2$-connected if, and only if, for any decomposition
$\dd=\dd_{1}+\ldots+\dd_{r}$, $b(Q)>b(\supp(\dd_{1}))+\ldots+b(\supp(\dd_{r}))$.
The statement then follows from \cite[Thm. 1.2.]{CB01}. From this
we deduce:

\begin{cor} \label{Cor/RatSgToric}

Let $Q$ be a $2$-connected quiver and $\dd=\underline{1}$. Then
$\mu_{Q,\dd}^{-1}(0)$ has rational singularities.

\end{cor}

\section{Rational singularities and totally negative quivers \protect\label{Sect/MainRes}}

In this section we prove Theorem \ref{Thm/MainResIntro}. Let us recall
the statement:

\begin{thm} \label{Thm/MainRes}

Let $Q$ be a quiver and $\dd\in\mathbb{N}^{Q_{0}}\setminus\{0\}$
such that $(Q,\dd)$ has property (P) . Then $\mu_{Q,\dd}^{-1}(0)$
has rational singularities.

\end{thm}

We follow the inductive strategy developed by Budur in \cite{Bud21}.
Let us use notations from \cite{Bud21}: we write, for short, $X(Q,\dd)=\mu_{Q,\dd}^{-1}(0)\subseteq R(\overline{Q},\dd)$,
$M(Q,\dd)=X(Q,\dd)\git\GL(\dd)$ and $Z(Q,\dd)=\left(\mu_{Q,\dd}^{-1}(0)\right)_{\tau_{\min}}$,
where $\tau_{\min}=\tau_{\min,\dd}$ is the semisimple type of $0\in X(Q,\dd)$
- see Section \ref{Subsect/MomMap}. We denote by $q:X(Q,\dd)\rightarrow M(Q,\dd)$
the quotient morphism. Recall that we defined (P) the following property
of $(Q,\dd)$: $Q$ is totally negative and if $\supp(\dd)$ has two
vertices joined by exactly one arrow, then $\dd\ne\underline{1}$.
Budur's reasoning is summed up in the following result.

\begin{thm}{\cite[Thm. 3.6.]{Bud21}} \label{Thm/BudurInduction}

Let $\mathcal{M}$ be a class of pairs $(Q,\dd)$, where $Q$ is a
quiver and $\dd\in\mathbb{N}^{Q_{0}}$ is a dimension vector. Suppose
that:
\begin{enumerate}
\item The class $\mathcal{M}$ is stable under the operation of building
pairs $(Q_{\tau},\ee)$ as in Section \ref{Subsect/EtaleSlices},
for $\tau$ a semisimple type occuring in $X(Q,\dd)$ and $(Q,\dd)\in\mathcal{M}$,
\item For every $(Q,\dd)\in\mathcal{M}$, the variety $X(Q,\dd)$ contains
a simple point and $\langle\dd,\dd\rangle<1$,
\item For every $(Q,\dd)\in\mathcal{M}$ such that $X(Q,\dd)$ contains
strictly semisimple points\footnote{In other words, $\dd\ne\epsilon_{i}$ for all $i\in Q_{0}$, as $\tau_{\min}=(\epsilon_{i},d_{i},\ i\in\supp(\dd))$.},\[
\dim q^{-1}(q(0))<2\cdot(1-\langle\dd,\dd\rangle-\sharp\{\text{loops in } Q_0\}).
\]
\end{enumerate}
Then for every $(Q,\dd)\in\mathcal{M}$, $X(Q,\dd)$ has rational
singularities.

\end{thm}

Let us consider $\mathcal{M}$ the class of pairs $(Q,\dd)$ satisfying
property (P) and $\supp(\dd)=Q$. $\mathcal{M}$ is preserved under
taking auxiliary quivers (this is an easy consequence of \cite[Prop. 2.11.]{Bud21}),
so in the above theorem, Assumption 1 is verified. Assumption 2 follows
from total negativity and Proposition \ref{Prop/TotNegSimp}. However,
Assumption 3 is satisfied for most, but not all dimension vectors.
This is the content of the following Lemma and Remark.

\begin{lem} \label{Lem/DimBoundTotNeg}

Let $\tau=\tau_{\min}$. Suppose that $(Q,\dd)$ has property (P).
Then either $\dd=\underline{1}$ or:\[
\dim q^{-1}(q(0))<2\cdot(1-\langle\dd,\dd\rangle-\sharp\{\text{loops in } Q_0\}).
\]\end{lem}

\begin{rmk} \label{Rmk/BoundFailure}

When $\dd=\underline{1}$, we may have:\[
\dim q^{-1}(q(0))=2\cdot(1-\langle\dd,\dd\rangle-\sharp\{\text{loops in } Q_0\}).
\]This is the case, for instance, when $Q$ has two vertices with two
loops each and joined by two arrows (regardless of their orientation).
This quiver arises as an auxiliary quiver for the quiver with one
vertex and two loops. Actually, within the class of quivers with dimension
vectors considered in \cite[Prop. 2.26.]{Bud21}, Assumption 3 fails
exactly for that pair. This is due to a computational gap in the proof
of \cite[Prop. 2.23.]{Bud21}. More precisely, on the first line of
the proof of \cite[Prop. 2.23.]{Bud21}, the right-hand side should
be $2(g-1)(n^{2}-\sum_{i}\beta_{i})-2r+2$ instead of $2(g-1)(n^{2}-\sum_{i}\beta_{i})$.

\end{rmk}

Therefore, we show by other means that $X(Q,\dd)$ has rational singularities
when $\dd=\underline{1}$. In order to incorporate those cases into
Budur's inductive argument, we prove the following modified version
of Theorem \ref{Thm/BudurInduction}:

\begin{thm} \label{Thm/ModifInduction}

Let $\mathcal{M}$ be a class of pairs $(Q,\dd)$, where $Q$ is a
quiver and $\dd\in\mathbb{N}^{Q_{0}}$ is a dimension vector. Let
us make the following assumptions:
\begin{enumerate}
\item The class $\mathcal{M}$ is stable under the operation of building
pairs $(Q_{\tau},\ee)$ as in Section \ref{Subsect/EtaleSlices},
for $\tau$ a semisimple type occuring in $X(Q,\dd)$ and $(Q,\dd)\in\mathcal{M}$,
\item For every $(Q,\dd)\in\mathcal{M}$, the variety $X(Q,\dd)$ contains
a simple point,
\item For every $(Q,\dd)\in\mathcal{M}$, suppose either that $X(Q,\dd)$
has rational singularities or that the following inequality holds:\[
\dim q^{-1}(q(0))<2\cdot(1-\langle\dd,\dd\rangle-\sharp\{\text{loops in } Q_0\}).
\]
\end{enumerate}
Then for every $(Q,\dd)\in\mathcal{M}$, $X(Q,\dd)$ has rational
singularities.

\end{thm}

Note that we removed the assumption $\langle\dd,\dd\rangle<1$ from
Assumption 2. This is harmless, for when $X(Q,\dd)$ contains a simple
point, $\langle\dd,\dd\rangle\leq1$ and equality only occurs when
$\dd=\epsilon_{i}$, for $i\in Q_{0}$ a vertex without loops, in
which case $X(Q,\dd)$ is smooth. Indeed, if $X(Q,\dd)$ contains
a simple point and $\langle\dd,\dd\rangle=1$, then $M(Q,\dd)$ is
reduced to a point by Proposition \ref{Prop/GeoMomMap}. Thus $X(Q,\dd)$
does not contain strictly semisimple points, which can only happen
if $\dd=\epsilon_{i}$ for some $i\in Q_{0}$. The number of loops
at $i$ is then $1-\langle\epsilon_{i},\epsilon_{i}\rangle=1-\langle\dd,\dd\rangle=0$.
At any rate, since Theorem \ref{Thm/MainRes} only concerns totally
negative quivers, the assumption $\langle\dd,\dd\rangle<1$ is automatically
satisfied.

Before proving Theorem \ref{Thm/ModifInduction}, we recall the following
results from \cite{Bud21}:

\begin{lem}{\cite[Lem. 2.16.]{Bud21}} \label{Lem/BoundDimStratZ}

Let $Q$ be a quiver and $\dd\in\mathbb{N}^{Q_{0}}$ a dimension vector.
Assume that:\[
\dim q^{-1}(q(0))<2\cdot(1-\langle\dd,\dd\rangle-\sharp\{\text{loops in } Q_0\}).
\] Then $\dim Z(Q,\dd)<2(1-\langle\dd,\dd\rangle)$.

\end{lem}

\begin{lem}{\cite[Lem. 3.3. - Proof of Lem. 3.4.]{Bud21}} \label{Lem/JetsMomMap}

Let $\pi_{m}:X(Q,\dd)_{m}\rightarrow X(Q,\dd)$ be the truncation
of $m$-jets:\[
\pi_m^{-1}(0)\simeq
\left\{
\begin{array}{ll}
R(\overline{Q},\dd)\times X(Q,\dd)_{m-2} & ,\ m\geq2, \\
R(\overline{Q},\dd) & ,\ m=1.
\end{array}
\right.
\]Moreover, $\dim\pi_{m}^{-1}(Z(Q,\dd)\cap X(Q,\dd)_{\sg})\leq\dim Z(Q,\dd)+\dim\pi_{m}^{-1}(0)$
(see \cite[Proof of Lem. 3.4. - Eqn. (10)]{Bud21}).

\end{lem}

\begin{proof}[Proof of Theorem \ref{Thm/ModifInduction}]

Consider a pair $(Q,\dd)\in\mathcal{M}$ for which it is not already
assumed that $X(Q,\dd)$ has rational singularities. We proceed by
descending induction on semisimple types occuring in $X(Q,\dd)$ i.e.
we show that $X(Q_{\tau},\ee)$ has rational singularities for all
semisimple types $\tau$ occuring in $X(Q,\dd)$. When $\tau$ is
maximal (i.e. simple, by Assumption 2), $Q_{\tau}$ has one vertex
and $\ee=\underline{1}$, so $X(Q_{\tau},\ee)$ is smooth.

Let us prove the induction step. We apply Proposition \ref{Prop/MustCrit}
to $X(Q,\dd)$, which is a complete intersection by Assumption 2 and
Proposition \ref{Prop/GeoMomMap}. We show that:\[
\dim\pi_m^{-1}(X(Q,\dd)_{\sg})<(m+1)\cdot\dim X(Q,\dd).
\]By induction, we assume that $X(Q_{\tau},\ee)$ has rational singularities
for $\tau>\tau_{\min}$. By Proposition \ref{Prop/RatSgClosedOrb},
we get that the open subset $X(Q,\dd)\setminus Z(Q,\dd)$ has l.c.i.
rational singularities and we obtain by Proposition \ref{Prop/MustCrit}:\[
\dim\pi_m^{-1}(X(Q,\dd)_{\sg}\setminus Z(Q,\dd))<(m+1)\cdot\dim X(Q,\dd).
\]Now, consider $\tau=\tau_{\min}$, i.e. $(Q_{\tau},\ee)=(Q,\dd)$.
From Proposition \ref{Prop/GeoMomMap} and Lemmas \ref{Lem/DimBoundTotNeg},
\ref{Lem/BoundDimStratZ}, \ref{Lem/JetsMomMap}, we obtain:\[
\dim\pi_m^{-1}(Z(Q,\dd)\cap X(Q,\dd)_{\sg})\leq\dim\left(Z(Q,\dd)\right)+\dim\pi_{m}^{-1}(0)<\dim M(Q,\dd)+\dim \pi_{m}^{-1}(0).
\]We prove that $\dim\pi_{m}^{-1}(Z(Q,\dd)\cap X(Q,\dd)_{\sg})<(m+1)\cdot\dim X(Q,\dd)$
by induction on $m$. For $m=1$, Lemma \ref{Lem/JetsMomMap} gives:\[
\dim M(Q,\dd)+\dim \pi_{m}^{-1}(0)=\dim M(Q,\dd)+\dim R(\overline{Q},\dd)=2\cdot\left(1-\langle\dd,\dd\rangle+\dd\cdot\dd-\langle\dd,\dd\rangle\right)=2\dim X(Q,\dd).
\]For $m\geq2$, we obtain:\[
\begin{split}
\dim M(Q,\dd)+\dim \pi_{m}^{-1}(0) & =\dim M(Q,\dd)+\dim R(\overline{Q},\dd)+\dim X(Q,\dd)_{m-2}, \\
& \leq2\dim X(Q,\dd)+(m-1)\cdot\dim X(Q,\dd)=(m+1)\cdot\dim X(Q,\dd).
\end{split}
\]where the second inequality holds by induction on $m$ and using the
following inequalities at step $m-2$:\[
\left\{
\begin{array}{rcl}
\dim\pi_{m-2}^{-1}(X(Q,\dd)_{\sm}) & = & (m-1)\cdot\dim X(Q,\dd), \\
\dim\pi_{m-2}^{-1}(Z(Q,\dd)\cap X(Q,\dd)_{\sg}) & < & (m-1)\cdot\dim X(Q,\dd), \\
\dim\pi_{m-2}^{-1}(X(Q,\dd)_{\sg}\setminus Z(Q,\dd)) & < & (m-1)\cdot\dim X(Q,\dd).
\end{array}
\right.
\] Therefore, at step $m$, we obtain the following inequalities:\[
\left\{
\begin{array}{rcl}
\dim\pi_{m}^{-1}(Z(Q,\dd)\cap X(Q,\dd)_{\sg}) & < & (m+1)\cdot\dim X(Q,\dd), \\
\dim\pi_{m}^{-1}(X(Q,\dd)_{\sg}\setminus Z(Q,\dd)) & < & (m+1)\cdot\dim X(Q,\dd).
\end{array}
\right.
\]So we obtain $\dim\pi_{m}^{-1}(X(Q,\dd)_{\sg})<(m+1)\cdot\dim X(Q,\dd)$
for all $m\geq1$, which proves, by Musta\c{t}\v{a}'s criterion,
that $X(Q,\dd)$ has rational singularities. \end{proof}

To complete the proof, we need to prove Lemma \ref{Lem/DimBoundTotNeg}
and show that $X(Q,\dd)$ has rational singularities when $\dd=\underline{1}$.
The latter fact can easily be deduced from the results in Section
\ref{Subsect/RatSgJets}.

\begin{lem} \label{Lem/RatSgToric}

If $(Q,\dd)$ has property (P) and $\dd=\underline{1}$, then $X(Q,\dd)$
has rational singularities.

\end{lem}

\begin{proof}

This follows from Corollary \ref{Cor/RatSgToric} since $\dd$ has
all entries equal to $1$ and the graph underlying $Q$ is $2$-connected.
We could also use \cite[Thm. 2.11.]{HSS21}. \end{proof}

We now turn to the proof of Lemma \ref{Lem/DimBoundTotNeg}. We first
prove the following intermediate inequality:

\begin{lem} \label{Lem/DimBoundLoops}

Let $Q=S_{g}$ be the quiver with one vertex and $g\geq2$ loops and
$d\geq2$. Let $(j_{s},m_{s},\ 1\leq s\leq h)$ be a top type compatible
with $\tau_{\min}$. Then:\[
2\cdot(1-\langle d,d\rangle-g)-\left(d^2-1+1-\langle d,d\rangle+m_{2}m_{1}+\ldots m_{h}m_{h-1}-g\cdot(m_{1}^2+\ldots m_{h}^2)\right)\geq d-1.
\]\end{lem}

\begin{proof}

Rearranging terms, the left-hand side reads:\[
\begin{split}
& g\cdot\left(d^2+\sum_sm_s^2-2\right)-2(d^2-1)-(m_{2}m_{1}+\ldots m_{h}m_{h-1}) \\
\geq\ & 2\cdot\left(d^2+\sum_sm_s^2-2\right)-2(d^2-1)-(m_{2}m_{1}+\ldots m_{h}m_{h-1}) \\
\geq\ & 2\left(\sum_sm_s^2-1\right)-(m_{2}m_{1}+\ldots m_{h}m_{h-1}) \\
\geq\ &\sum_sm_s^2-1+\frac{1}{2}\cdot\left(m_1^2+m_h^2+(m_1-m_2)^2+\ldots+(m_{h-1}-m_h)^2-2\right) \\
\geq\ & d-1.
\end{split}
\]Note that both sides are equal when $g=2$ and $m_{1}=\ldots=m_{s}=1$.
\end{proof}

\begin{proof}[Proof of Lemma \ref{Lem/DimBoundTotNeg}]

Without loss of generality, we assume that $\supp(\dd)=Q$; otherwise
we restrict to $\text{supp}(\dd)$, which also has property (P). Let
$g_{i}$ be the number of loops of $Q$ at vertex $i$, $r_{ij}$
be the number of arrows between vertices $i\ne j$ and $(j_{s},m_{s},\ 1\leq s\leq h)$
be a top type compatible with $\tau$. Then, by Proposition \ref{Prop/CBDimBound},
the inequality above holds if the following holds for arbitrary top-type:\[
\sum_id_i^2-1+1-\langle\dd,\dd\rangle+\sum_sm_sz_s-\sum_sm_s^2g_{j_s}<2\cdot(1-\langle\dd,\dd\rangle-\sum_ig_i).
\]We now want to split this inequality along vertices of $Q$. Let us
first rearrange the indices $1\leq s\leq h$ so that $j_{1}=\ldots=j_{s_{1}}=1$,
and so on until $j_{s_{r-1}+1}=\ldots=j_{s_{r}}=r$. Then $z_{s_{i-1}+1}=0$
and $z_{s_{i-1}+j+1}=m_{s_{i-1}+j}$ for $j>0$, $1\leq i\leq r$,
with the convention that $s_{0}=0$. We obtain:\[
\sum_id_i^2-1+1-\langle\dd,\dd\rangle+\sum_i(m_{s_{i-1}+2}m_{s_{i-1}+1}+\ldots m_{s_i}m_{s_i-1})-\sum_ig_i\cdot(m_{s_{i-1}+1}^2+\ldots m_{s_i}^2)<2\cdot(1-\langle\dd,\dd\rangle-\sum_ig_i).
\]We know from Lemma \ref{Lem/DimBoundLoops} that for all $1\leq i\leq r$:\[
\begin{split}
& 2\cdot(1-\langle d_i\epsilon_i,d_i\epsilon_i\rangle-g_i) \\
& -\left(d_i^2-1+1-\langle d_i\epsilon_i,d_i\epsilon_i\rangle+m_{s_{i-1}+2}m_{s_{i-1}+1}+\ldots+ m_{s_i}m_{s_i-1}-g_i\cdot(m_{s_{i-1}+1}^2+\ldots m_{s_i}^2)\right) \\
\geq\ & d_i-1,
\end{split}
\]where equality holds for top type $(m_{s}=1,\ 1\leq s\leq h)$. Taking
the remaining terms in the inequality (right-hand side minus left-hand
side) gives:\[
1-\langle\dd,\dd\rangle-\sum_i(1-\langle d_i\epsilon_i,d_i\epsilon_i\rangle)-(r-1)=\sum_{i\ne j}r_{ij}d_id_j-2(r-1).
\]Set $r_{1}$ (resp. $r_{2}$) the number of vertices $i\in Q_{0}$
such that $d_{i}=1$ (resp. $d_{i}\geq2$). Then $r=r_{1}+r_{2}$
(recall that we assume that $\supp(\dd)=Q$). Then:\[
\begin{split}
\sum_{i\ne j}r_{ij}d_id_j-2(r-1) & \geq 4 \cdot\frac{r_2(r_2-1)}{2}+2r_1r_2+\frac{r_1(r_1-1)}{2}-2(r-1) \\
& \geq 2(r_2-1)(r-1)+\frac{r_1(r_1-1)}{2}.
\end{split}
\]Summing everything, we obtain:\[
\begin{split}
& 2\cdot(1-\langle\dd,\dd\rangle-\sum_ig_i) \\
& -\left(\sum_id_i^2-1+1-\langle\dd,\dd\rangle+\sum_i(m_{s_{i-1}+2}m_{s_{i-1}+1}+\ldots m_{s_i}m_{s_i-1})-\sum_ig_i\cdot(m_{s_{i-1}+1}^2+\ldots m_{s_i}^2)\right) \\
\geq & \sum_i(d_i-1)+2(r_2-1)(r-1)+\frac{r_1(r_1-1)}{2}.
\end{split}
\]If $\dd>\underline{1}$, the right-hand side is positive, so we get
the desired inequality. \end{proof}

Theorem \ref{Thm/MainRes} now follows from Lemmas \ref{Lem/DimBoundTotNeg},
\ref{Lem/RatSgToric} and Theorem \ref{Thm/ModifInduction}.

\begin{rmk} 

Unfortunately, it may happen that $X(Q,\dd)$ has rational singularities,
while the dimension bound from \cite[Lem. 2.16.]{Bud21} fails. This
is the case, for instance, with the following quivers: $\bullet\rightrightarrows\bullet$
and $\bullet\rightrightarrows\bullet\rightrightarrows\bullet$, when
$\dd=\underline{1}$ (both are 2-connected).

\end{rmk}

\section{Applications to moduli of totally negative 2-Calabi-Yau categories
\protect\label{Sect/App2CY}}

\subsection{Rational singularities \protect\label{Subsect/RatSgTotNeg2CY}}

Let $\mathfrak{M}=[X/G]$ be a quotient stack coming from a 2-Calabi-Yau
category $\mathcal{T}$ and whose closed points parametrize objects
in an abelian subcategory $\mathcal{A}\subseteq\mathcal{T}$ (see
Definition \ref{Def/QuotStack2CY}). As a consequence of Theorem \ref{Thm/MainRes},
if all auxiliary quivers arising from $\mathfrak{M}$ have property
(P), then $X$ has l.c.i. and rational singularities. This leads us
to the notion of a totally negative category, as introduced in \cite{DHSM22}.
We define it using Hom-spaces in triangulated categories, similarly
to Definition \ref{Def/CYCat}, since in the example we consider,
Ext-quivers can be computed from such Hom-spaces (see Section \ref{Subsect/Mod2CY}).

\begin{df} \label{Def/TotNegCat}

Let $\mathcal{C}$ be a full subcategory of a $\KK$-linear, Hom-finite,
triangulated category $\mathcal{T}$. Suppose that for all $F_{1},F_{2}\in\mathcal{T}$,
$\Hom(F_{1},F_{2}[i])=0$ for all but finitely many $i\in\mathbb{Z}$.
$\mathcal{C}$ is called totally negative if, for all $F_{1},F_{2}\in\mathcal{C}$:\[
\sum_{i\geq0}(-1)^i\cdot\hom(F_1,F_2[i])<0.
\]\end{df}

For example, when $\Pi_{Q}$ or $\Lambda^{q}(Q)$ are 2-Calabi-Yau,
the categories of $\Pi_{Q}$-modules or $\Lambda^{q}(Q)$-modules
are totally negative if, and only if, $Q$ is totally negative (see
Section \ref{Subsect/Mod2CY}). Here, we see the category $\mathcal{C}$
of $\Pi_{Q}$-modules (or $\Lambda^{q}(Q)$-modules) as a subcategory
of its bounded derived category $\mathcal{T}$, which is 2-Calabi-Yau
in the sense of Definition \ref{Def/CYCat}.

Note that, in order to prove that $X$ has l.c.i., rational singularities,
we only need to show that the auxiliary quivers arising from $X$
are totally negative. In other words, we only need to show that the
subcategory $\mathcal{C}\subseteq\mathcal{A}$ generated by simple
summands occuring at closed points of $\mathfrak{M}$ is totally negative.
We also require an additional property on the subset of simple objects
so that auxiliary quivers also satisfy property (P).

\begin{thm} \label{Thm/RatSgTotNeg2CY}

Let $\mathfrak{M}=[X/G]$ be a quotient stack coming from a 2-Calabi-Yau
category and $M:=X\git G$. Assume that:
\begin{enumerate}
\item For any closed point $x\in\mathfrak{M}(\KK)$, the Ext-quiver associated
to $x$ is the double of a totally negative quiver,
\item Simple objects form a dense subset of $M$.
\end{enumerate}
Then $X$ is locally complete intersection and has rational singularities.

\end{thm}

\begin{proof}

Let $x\in X(\KK)$ be a point with closed orbit and $(\overline{Q'},\ee)$
the associated Ext-quiver. Then by Theorem \ref{Thm/LocMod2CY}, there
exists an affine $\GL(\ee)$-variety $W$, with a fixed point $w$
and a commutative diagram:\[
\begin{tikzcd}[ampersand replacement = \&]
(\mu_{Q',\ee}^{-1}(0)\times^{\GL(\ee)}G,[0,\Id]) \ar[d] \& (W\times^{\GL(\ee)}G,[w,\Id]) \ar[l]\ar[r]\ar[d] \& (X,x) \ar[d] \\
(M_{\Pi_{Q'},\ee},0) \& (W\git\GL(\ee),w) \ar[l]\ar[r] \& (M,x)
\end{tikzcd}
\]such that the horizontal maps are étale and the squares are cartesian.
Since the set of simple objects is dense in $M$, one can find a simple
object in the image of $(W\git\GL(\ee),w)\rightarrow(M,x)$. Moreover,
strongly étale morphisms preserve stabilizers (see \cite[Lem. 2.6.]{Bud21}),
so we can find a semisimple point in $\mu_{Q',\ee}^{-1}(0)$ whose
stabilizer is $\KK^{\times}$ i.e. there is a simple $\Pi_{Q'}$-module
of dimension $\ee$. From this and assumption 1., we deduce that $Q'$
satisfies property (P).

Therefore, by Theorem \ref{Thm/MainRes} and Lemmas \ref{Lem/LciSgDesc},
\ref{Lem/RatSgDesc}, any semisimple point of $X$ has a $G$-saturated
neighborhood which is locally complete intersection and has rational
singularities. Since $X$ is covered by such neighborhoods, the conclusion
follows. \end{proof}

We now give a few examples and apply our result to moduli of coherent
sheaves on K3 surfaces and representations of multiplicative preprojective
algebras.

\subsubsection*{$\Lambda^{q}(Q)$-modules}

In the case of the multiplicative preprojective algebra $\Lambda^{q}(Q)$,
the auxiliary quivers are computed as in the additive case from $Q$
itself. Therefore, if $Q$ is totally negative, then all auxiliary
quivers are. The variety $R(\Lambda^{\alpha}(Q),\dd)$ is also known
to be an affine complete intersection, by \cite[Thm. 1.11.]{CBS06}.
Thus, we obtain:

\begin{cor} \label{Cor/RatSgMultPreproj}

Let $Q$ be a quiver and $\dd\in\mathbb{N}^{Q_{0}}\setminus{0}$ such
that $(Q,\dd)$ has property (P). Then $R(\Lambda^{q}(Q),\dd)$ is
a complete intersection and has rational singularities.

\end{cor}

\begin{proof}

We apply Theorem \ref{Thm/RatSgTotNeg2CY}. It remains to check assumption
2. Since $(Q,\dd)$ has property (P), there exists a simple $\Pi_{Q}$-module
of dimension $\dd$. Thus by \cite[Thm. 1.2.]{CB01} and \cite[Thm. 1.11.]{CBS06},
simple points form an open dense subset of $R(\Lambda^{q}(Q),\dd)$,
hence a dense subset of $M_{\Lambda^{q}(Q),\dd}$. \end{proof}

\subsubsection*{Semistable coherent sheaves on a K3 surface}

In the case of sheaves on a K3 surface, controlling Ext-quivers of
all polystable sheaves with given Mukai vector is more delicate. In
\cite{BZ19}, it was already mentioned that for many polystable sheaves,
one obtains the auxiliary quivers of the $g$-loop quiver. We spell
out conditions for this to be true. However, one can also obtain étale-local
models which are much more difficult to study. We illustrate this
with the case of one dimensional sheaves on an elliptic K3 surface.

Let $(S,H)$ be a complex, projective, polarized K3 surface. Given
a polystable sheaf $\mathcal{F}=\bigoplus_{i=1}^{r}\mathcal{F}_{i}^{\oplus e_{i}}$,
with distinct stable summands $\mathcal{F}_{i}$ of Mukai vectors
$\textbf{v}_{i}$, recall that the Ext-quiver $\overline{Q'}$ has
the following number of arrows from $i$ to $j$:\[
\ext^1(\mathcal{F}_{i},\mathcal{F}_{j})=
\left\{
\begin{array}{ll}
2+\textbf{v}_{i}\cdot\textbf{v}_{i}, & \text{ if } i=j, \\
\textbf{v}_{i}\cdot\textbf{v}_{j}, & \text{ if } i\ne j.
\end{array}
\right.
\]$Q'$ is totally negative if, and only if, $\textbf{v}_{i}\cdot\textbf{v}_{j}>0$
for all $i,j$. To see this, one can simply check when $Q'$ has at
least two loops at each vertex and one arrow joining any pair of vertices
(recall from Section \ref{Subsect/Mod2CY} that $\textbf{v}_{i}\cdot\textbf{v}_{i}$
is even). Otherwise, one can also notice that:\[
(\dd,\ee)_{Q'}=-\left(\sum_i d_i\textbf{v}_{i}\right)\cdot\left(\sum_i e_i\textbf{v}_{i}\right).
\]We now give examples of auxiliary quivers arising from $\mathfrak{M}_{S,H}(\textbf{v})$.
We rely on criteria for the existence of (semi)stable sheaves of a
given Mukai vector, due to Yoshioka \cite[Thm. 8.1.]{Yos01} (see
also \cite[\S2.4.]{KLS06}).

\begin{df} \label{Def/PrimPosMukVec}

A Mukai vector $\textbf{v}$ is called primitive if it cannot be written
$\textbf{v}=m\textbf{v}_{0}$, for some other Mukai vector $\textbf{v}_{0}$
and $m\geq2$.

A primitive vector $\textbf{v}=(r,\textbf{c},a)$ is called positive
if $\textbf{v}\cdot\textbf{v}\geq-2$ and one of the following holds:
\begin{itemize}
\item $r>0$ and $\textbf{c}\in\mathrm{NS}(S)$,
\item $r=0$, $\textbf{c}\in\mathrm{NS}(S)$ is effective and $a\ne0$,
\item $r=0$, $\textbf{c}=0$ and $a>0$.
\end{itemize}
\end{df}

Let $\text{Amp}(S)$ be the space of polarizations of $S$ and $\textbf{v}$
a primitive, positive Mukai vector. One can associate walls in $\text{Amp}(S)\otimes_{\mathbb{Z}}\RR$
to $\textbf{v}$ i.e. subspaces of real codimension one, such that,
for $H$ lying outside these walls, there are no strictly semistable
sheaves with Mukai vector $\textbf{v}$ (see \cite[\S1.4.]{Yos01}
for $\rk(\textbf{v})=0$ and \cite[Ch. 4.C.]{HL10} for $\rk(\textbf{v})>0$).
Those polarizations are called $\textbf{v}$-generic. The existence
results for semistable sheaves are summed up in the following:

\begin{prop}{\cite[Thm. 2.2-4.]{AS18}} \label{Prop/ExisStabShK3}

If $\textbf{v}$ is a primitive, positive Mukai vector, then $M(\textbf{v})\ne\emptyset$.
If moreover $H$ is $\textbf{v}$-generic, then $M(\textbf{v})=M^{s}(\textbf{v})$.

\end{prop}

In general, there is no reason for auxiliary quivers of all polystable
sheaves in $M(\textbf{v})$ to be totally negative. However, if we
choose a moduli space $M(\textbf{v})$ of high enough dimension and
a generic polarization, then the auxiliary quiver of any polystable
sheaf in $M(\textbf{v})$ is totally negative (and it arises from
a $g$-loop quiver). This was already observed in \cite{BZ19}.

\begin{prop} \label{Prop/ExtQuivK3}

Let $\textbf{v}=m\textbf{v}_{0}$, where $\textbf{v}_{0}$ is a primitive,
positive Mukai vector and assume that $H$ is $\textbf{v}_{0}$-generic.
Suppose also that $\textbf{v}\cdot\textbf{v}>0$. Then for any polystable
sheaf $\mathcal{F}=\bigoplus_{i}\mathcal{F}_{i}^{\oplus e_{i}}$ with
Mukai vector $\textbf{v}$, the auxiliary quiver of $\mathcal{F}$
is an auxiliary quiver of some $g$-loop quiver, where $g\geq2$.

\end{prop}

\begin{proof}

Call $\mathbf{v}_{i}$ the Mukai vector of $\mathcal{F}_{i}$. We
leave out the obvious case where $\mathcal{F}$ is stable. Since $H$
is $\textbf{v}_{0}$-generic and by construction of walls, there exist
$r_{i}\in\mathbb{Q}$ such that $\mathbf{v}_{i}=r_{i}\mathbf{v}$
(see \cite[\S2.4.]{KLS06}). Since $\mathbf{v}_{0}$ is primitive,
$\mathbf{v}_{i}=m_{i}\mathbf{v}_{0}$ for some $m_{i}\in\mathbb{Z}$.
Moreover, $\mathbf{v}_{0}\cdot\mathbf{v}_{0}$ is a positive even
integer, hence $\mathbf{v}_{0}\cdot\mathbf{v}_{0}=2g-2$, with $g\geq2$.
Consequently, the auxiliary quiver of $\mathcal{F}$ has $1+m_{i}^{2}(g-1)$
loops at each vertex and $2m_{i}m_{j}(g-1)$ arrows joining vertices
$i\ne j$. One recognizes the auxiliary quiver of the $g$-loop quiver
for the semisimple type $(m_{i},e_{i}\ ;\ 1\leq i\leq r)$ (see \cite[Prop. 2.26.]{Bud21}).
\end{proof}

In that case, \cite[Thm. 1.1.]{Bud21} yields the following corollary,
already observed in \cite{BZ19}:

\begin{cor}

Under the hypotheses of Proposition \ref{Prop/ExtQuivK3}, $U_{\textbf{v},m_{\textbf{v}}}^{\mathrm{ss}}$
is locally complete intersection and has rational singularities.

\end{cor}

In \cite[\S6.]{AS18}, Arbarello and Saccà analyzed Ext-quivers of
a pure, one-dimensional torsion sheaf $\mathcal{F}$ with primitive,
positive Mukai vector $\textbf{v}=(0,\textbf{[D]},\chi)$ ($D$ effective,
$\chi\ne0$). Write $D=\sum_{i\in I}n_{i}D_{i}$ where $n_{i}\geq0$
and $D_{i}$ is integral. Then the auxiliary quiver of $\mathcal{F}$
has vertex set $I$, $g(D_{i})$ loops at vertex $i$ and $D_{i}\cdot D_{j}$
arrows joining vertices $i$ and $j$. The following example illustrates
how difficult the geometry of étale-local models of $\mathfrak{M}_{S,H}(\textbf{v})$
gets.

\begin{exmp}

Let $\textbf{v}=(0,\textbf{[D]},1)$, where $D$ is an elliptic curve
in $S$ (for instance, suppose $S$ is an elliptic K3 surface and
$D$ is a generic fibre). Then $\textbf{v}$ is primitive, positive,
$\textbf{v}\cdot\textbf{v}=2g(D)-2=0$ and taking a $\textbf{v}$-generic
polarization, we obtain by Proposition \ref{Prop/ExisStabShK3} a
stable sheaf $\mathcal{F}$ with Mukai vector $\textbf{v}$. For any
$n\geq1$, the auxiliary quiver of $\mathcal{F}^{\oplus n}$ is the
Jordan quiver (one vertex with one loop), with dimension vector $n$.
The associated local model is the commuting scheme:\[
C_2:=V([x,y]_{i,j},\ 1\leq i,j\leq n)\subset\Mat(n,\KK)_x\times\Mat(n,\KK)_y
\]Commuting schemes have been studied by many authors and were only
recently proved to be normal and Cohen-Macaulay (see \cite{Cha20}
and the references therein). Whether commuting schemes have rational
singularities is a more difficult question and, as far as we know,
this problem is still open.

\end{exmp}

\subsection{Counts of jets and p-adic integrals \protect\label{Subsect/ArithmCor}}

In this section, we go back to our initial arithmetic problem on counts
of jets. We deduce two consequences of Theorem \ref{Subsect/RatSgTotNeg2CY}.
First, we give a partial answer to our original question on counts
of quiver representations over $\mathbb{F}_{q}[t]/(t^{n})$. We then
consider a general quotient stack $[X/G]$ coming from a totally negative
2-Calabi-Yau category and show that the counts of jets on $X$ converge
to the p-adic volume of a certain analytic manifold associated to
$X$, following \cite{AA16} and \cite{Yas17,COW24}.

\paragraph*{Convergence for counts of jets}

Let $Q$ be a quiver and $\dd\in\mathbb{N}^{Q_{0}}$. If $(Q,\dd)$
has property (P), we already know from Proposition \ref{Prop/GeoMomMap}
that $\mu_{Q,\dd}^{-1}(0)$ is a complete intersection. Thus we obtain
from Proposition \ref{Prop/JetCountLim}:

\begin{cor} \label{Cor/TotNegJetLim}

Let $Q$ be a quiver and $\dd\in\mathbb{N}^{Q_{0}}\setminus\{0\}$
such that $(Q,\dd)$ has property (P). Then for almost all primes
$p$ and for all finite fields $\mathbb{F}_{q}$ of characteristic
$p$, $q^{-n\dim\mu_{Q,\dd}^{-1}(0)}\cdot\sharp\mu_{Q,\dd}^{-1}(0)(\mathbb{F}_{q}[t]/(t^{n}))$
converges when $n$ goes to infinity.

\end{cor}

Now that we have established the existence of $\underset{n\rightarrow+\infty}{\lim}q^{-n\dim\mu_{Q,\dd}^{-1}(0)}\cdot\sharp\mu_{Q,\dd}^{-1}(0)(\mathbb{F}_{q}[t]/(t^{n}))$
for all totally negative quivers, it is natural to ask whether this
limit is uniform over all finite fields $\mathbb{F}_{q}$, as shown
by Wyss in the case $\dd=\underline{1}$.

\begin{cj}

Let $Q$ be a quiver and $\dd\in\mathbb{N}^{Q_{0}}\setminus\{0\}$
such that $(Q,\dd)$ has property (P). There exists a rational fraction
$W\in\mathbb{Q}(T)$ such that, for almost all primes $p$ and for
all finite fields $\mathbb{F}_{q}$ of characteristic $p$:\[
W(q)=\underset{n\rightarrow+\infty}{\lim}q^{-n\dim \mu_{Q,\dd}^{-1}(0)}\cdot\sharp\mu_{Q,\dd}^{-1}(0)(\mathbb{F}_{q}[t]/(t^{n})).
\]\end{cj}

\paragraph*{Counts of jets and p-adic volumes of mildly singular schemes}

Another natural question is whether the quantity $\underset{n\rightarrow+\infty}{\lim}q^{-n\dim\mu_{Q,\dd}^{-1}(0)}\cdot\sharp\mu_{Q,\dd}^{-1}(0)(\mathbb{F}_{q}[t]/(t^{n}))$
has a geometric interpretation. We adress this by showing that the
limit can be interpreted as a $p$-adic integral, using earlier work
of Aizenbud-Avni \cite{AA16}. Our result holds for a wider class
of schemes $X$, including atlases of quotient stacks $[X/G]$ parametrizing
objects of totally negative 2-Calabi-Yau categories. We then identify
this p-adic integral to the canonical volume of $X$ built, for instance,
in \cite{Yas17,COW24}.

Consider $F$ some finite extension of $\mathbb{Q}_{p}$, with residue
field $\mathbb{F}_{q}$ and valuation ring $\mathcal{O}\subset F$.
Given a scheme $X$ which is Gorenstein, of pure dimension $d$ over
$\mathcal{O}$, one can construct a canonical measure $\mu_{\mathrm{can}}$
on $X^{\natural}:=X^{\mathrm{sm}}(F)\cap X(\mathcal{O})$. This canonical
measure is built by glueing measures obtained from gauge forms (i.e.
nowhere vanishing differential forms of top degree) on $X^{\mathrm{sm}}$
(the smooth locus of $X$ relative to $\mathcal{O}$) - see \cite[\S 4]{Yas17}
or \cite[\S 3]{COW24}. Let us briefly recall how this construction
works.

Suppose that the structure morphism $X\rightarrow\Spec(\mathcal{O})$
is Gorenstein, of pure dimension $d$ (see \cite[\href{https://stacks.math.columbia.edu/tag/0C02}{Tag 0C02}]{SP}).
Then $X$ admits a canonical invertible sheaf $\Omega_{X/\mathcal{O}}$,
which restricts to $\Omega_{X^{\mathrm{sm}}/\mathcal{O}}^{d}$ on
the smooth locus $X^{\mathrm{sm}}$ (see \cite[\S 3.5.]{Con00}).
Consider trivializing open subsets $V_{k}\subseteq X$ for $\Omega_{X/\mathcal{O}}$
and non-vanishing sections $\omega_{k}\in\Gamma(V_{k},\Omega_{X/\mathcal{O}})$.
Note that $X^{\natural}=\bigcup_{k}V_{k}^{\natural}$, since for $x\in X(\mathcal{O})$,
if $x_{\mathbb{F}_{q}}\in V_{k}(\mathbb{F}_{q})$, then $x\in V_{k}(\mathcal{O})$.
The sections $\omega_{k}$ induce measures $\nu_{k}$ on $V_{k}^{\natural}$
- see \cite[\S 3.1.]{AA16} or \cite[\S 3]{COW24} for the construction
of these measures. Moreover, for $x\in V_{k}^{\natural}\cap V_{l}^{\natural}$,
$\frac{\omega_{k}}{\omega_{l}}(x)\in\mathcal{O}^{\times}$. The change
of variables formula for $p$-adic integrals then implies that the
measures $\nu_{k}$ glue to a measure $\mu_{\mathrm{can}}$ which
does not depend on the choices of $\omega_{k}$.

Given a $\mathbb{Z}$-scheme $X$ such that $X_{\bar{\mathbb{Q}}}$
has l.c.i. and rational singularities, we show that the counts of
jets we consider converge to $\mu_{\mathrm{can}}(X^{\natural})$,
when the residual characteristic of $F$ is large enough. Let us first
state an intermediate result for appropriate $\mathcal{O}$-schemes:

\begin{lem} \label{Lem/JetCountCanMeas}

Let $X$ be a finite type $\mathcal{O}$-scheme. Assume that $X$
is flat over $\mathcal{O}$, with l.c.i. geometric fibers of pure
dimension $d$. Assume also that $X_{\bar{F}}$ has rational singularities.

Then there is a well-defined canonical measure $\mu_{\mathrm{can}}$
on $X^{\natural}=X^{\mathrm{sm}}(F)\cap X(\mathcal{O})$. Moreover,
the sequence $q^{-nd}\cdot\sharp X(\mathcal{O}/(\varpi^{n})),\ n\geq1$
converges and its limit is given by:\[
\underset{n\rightarrow +\infty}{\lim}\frac{\sharp X(\mathcal{O}/\varpi^{n})}{q^{nd}}
=
\mu_{\mathrm{can}}(X^{\natural}).
\]\end{lem}

\begin{proof}

Since $X\rightarrow\Spec(\mathcal{O})$ is flat, with l.c.i. geometric
fibers, it is a Gorenstein morphism (of pure dimension $d$, by assumption
- see \cite[\href{https://stacks.math.columbia.edu/tag/0C02}{Tag 0C02}]{SP}).
Then the construction above applies and $X^{\natural}$ can be endowed
with a canonical measure $\mu_{\mathrm{can}}$.

We now describe $X$ locally as a fiber of an appropriate morphism
$\varphi:\mathbb{A}_{\mathcal{O}}^{r}\rightarrow\mathbb{A}_{\mathcal{O}}^{m}$.
On the one hand, this allows us to rewrite counts of $\mathcal{O}/(\varpi^{n})$-points
as measures of $p$-adic balls with respect to the pushforward by
$\varphi$ of the Haar measure on $\mathcal{O}^{r}$, following \cite[\S 4.2.]{Wys17b}.
On the other hand, if $\varphi$ is a FRS morphism (see \cite[Def. II.]{AA16}),
we can apply a Fubini theorem by Aizenbud and Avni \cite[Thm. 3.16.]{AA16}
to compute this pushforward measure and relate it to the canonical
measure on $X$.

As $X\rightarrow\Spec(\mathcal{O})$ is flat, with l.c.i. geometric
fibers, we can cover $X$ with affine open subsets of the form: \[
U_i=\Spec(\mathcal{O}[T_1,\ldots,T_{r_i}]/(f_{1,i},\ldots,f_{m,i})),
\] whose nonempty fibers over $\mathcal{O}$ have dimension $d=r_{i}-m_{i}$
(see \cite[\href{https://stacks.math.columbia.edu/tag/01UB}{Tag 01UB}]{SP}).
Observe that $U_{i}$ is of pure dimension $d$ over $\mathcal{O}$
if, and only if, $(U_{i})_{\mathbb{F}_{q}}\ne\emptyset$, as $U_{i}$
is flat over $\mathcal{O}$. Moreover, $X^{\natural}=\bigcup_{i}U_{i}^{\natural}$
as noted above, and $\sharp X(\mathcal{O}/(\varpi^{n}))$ can be obtained
from $\sharp U_{i}(\mathcal{O}/(\varpi^{n}))$ by inclusion-exclusion
(see \cite[Lem. 3.1.1.]{AA18}). Therefore, the $U_{i}$ whose contribution
to $\mu_{\mathrm{can}}(X^{\natural})$ and $\sharp X(\mathcal{O}/(\varpi^{n}))$
is non-zero are of pure dimension $d$ over $\mathcal{O}$ and we
may restrict to those for the rest of the proof.

Let $U_{i}$ be such an open subset. We temporarily drop the subscript
$i$ to ease notations. Consider $\varphi:=(f_{1},\ldots,f_{m}):\mathbb{A}_{\mathcal{O}}^{r}\rightarrow\mathbb{A}_{\mathcal{O}}^{m}$.
The locus $M=\{x\in\mathbb{A}_{\mathcal{O}}^{r}\ \vert\ \dim_{x}(\varphi^{-1}(\varphi(x)))\leq d\}$
is open, by Chevalley's semicontinuity theorem \cite[Thm. 13.1.3.]{Gro66},
and contains $U=\varphi^{-1}(0)$. Moreover, $\dim_{x}(\varphi^{-1}(\varphi(x)))\geq r-m=d$
by construction, so $\dim_{x}(\varphi^{-1}(\varphi(x)))=d$ for all
$x\in M$. Therefore, $\varphi\vert_{M}$ is flat (see \cite[Prop. 6.1.5.]{Gro65})
and $N:=\varphi(M)\subseteq\mathbb{A}_{\mathcal{O}}^{m}$ is open.
We have thus obtained a morphism $\varphi:M\rightarrow N$ which is
flat, with l.c.i. geometric fibers of pure dimension $d$, and such
that $U=\varphi^{-1}(0)$.

Let us now relate counts of $\mathcal{O}/(\varpi^{n})$-points on
$U$ to the pushforward by $\varphi$ of the Haar measure on $\mathcal{O}^{r}$.
Let $\omega_{M}\in\Gamma(M,\Omega_{M/\mathcal{O}}^{r})$ (resp. $\omega_{N}\in\Gamma(N,\Omega_{N/\mathcal{O}}^{m})$)
be a gauge form and $\mu_{M}$ (resp. $\mu_{N}$) the associated measure
on $M^{\natural}$ (resp. $N^{\natural}$). We also call $\varphi:M^{\natural}\rightarrow N^{\natural}$
the map of $F$-analytic manifolds induced by $\varphi$. Then $\varpi^{n}\mathcal{O}^{m}\subseteq N^{\natural}$
and\[
\frac{\sharp U(\mathcal{O}/(\varpi^n))}{q^{nd}}
=
\sharp\{x\in M(\mathcal{O}/(\varpi^n))\ \vert\ \varphi(x)=0\in(\mathcal{O}/(\varpi^n))^m\}\cdot\frac{q^{-nr}}{q^{-nm}}
=
\frac{\varphi_*\mu_M(\varpi^n\mathcal{O}^m)}{\mu_N(\varpi^n\mathcal{O}^m)}.
\]As shown in \cite[\S 3.3.]{AA16}, the measure $\varphi_{*}(\mu_{M})$
is absolutely continuous with respect to $\mu_{N}$ and its density
at $y\in N^{\natural}$ is given by a $p$-adic integral on $\varphi^{-1}(y)^{\natural}$.
We now build explicit gauge forms inducing the corresponding measures
on $\varphi^{-1}(y)^{\natural}$. As before, $\varphi$ is a Gorenstein
morphism of pure dimension $d$, so there exists a canonical invertible
sheaf $\Omega_{M/N}$, which restricts to $\Omega_{M^{\mathrm{sm}}/N}^{d}$
on the smooth locus $M^{\mathrm{sm}}$ of $\varphi$. Moreover, there
is an isomorphism of invertible sheaves $\Omega_{M/\mathcal{O}}^{r}\simeq\Omega_{M/N}\otimes\varphi^{*}\Omega_{N/\mathcal{O}}^{m}$,
by \cite[Thm. 4.3.3.]{Con00}. From this we obtain a nowhere-vanishing
section $\eta\in\Gamma(M,\Omega_{M/N})$ such that $\omega_{M}=\eta\otimes\varphi^{*}\omega_{N}$
and which restricts to a nowhere-vanishing section $\eta_{y}\in\Gamma(\varphi^{-1}(y),\Omega_{\varphi^{-1}(y)/\mathcal{O}})$
for any $y\in N^{\natural}$, by \cite[Thm. 3.6.1.]{Con00}. Let us
call $\mu_{y}$ the measure induced by the gauge form $(\eta_{y})\vert_{\varphi^{-1}(y)^{\mathrm{sm}}}$
on $\varphi^{-1}(y)^{\natural}$. When $y=0$, we simply write $\eta:=\eta_{0}$
and $\mu:=\mu_{0}$ for the induced measure on $U^{\natural}=\varphi^{-1}(0)^{\natural}$.
By construction, $\mu$ is the restriction of $\mu_{\mathrm{can}}$
to $U^{\natural}$.

The next step is to apply \cite[Thm. 3.16.]{AA16}. We first need
to restrict $\mu_{M}$ to a locus where $\varphi_{F}$ is FRS. Since
$\varphi_{F}$ is flat, the locus where geometric fibers of $\varphi_{F}$
have rational singularities is open - see  \cite[Thm. 4]{Elk78} -
and contains $U_{F}$. Denote by $Z_{F}\subseteq M_{F}$ the complement
of this locus. Then $Z_{F}(F)\cap U(F)=\emptyset$. Since $M^{\natural}$
and $N^{\natural}$ are compact, $\varphi$ is closed and there exists
some $n_{0}\geq0$ such that $Z_{F}(F)\cap\varphi^{-1}(\varpi^{n_{0}}\mathcal{O}^{m})=\emptyset$.
Consequently, $\varphi_{F}$ is FRS on $M_{F}\setminus Z_{F}$ and
$A:=\varphi^{-1}(\varpi^{n_{0}}\mathcal{O}^{m})$ is a compact subset
of $(M_{F}\setminus Z_{F})(F)$. By \cite[Thm. 3.16.]{AA16}, the
density of $\varphi_{*}\left((\mu_{M})\vert_{A}\right)$ with respect
to $\mu_{N}$ is given by the function $y\mapsto\mathbf{1}_{\varpi^{n_{0}}\mathcal{O}^{m}}(y)\cdot\mu_{y}(\varphi^{-1}(y)^{\natural})$,
which is continuous. Therefore, for $n\geq n_{0}$:\[
\frac{\sharp U(\mathcal{O}/(\varpi^n))}{q^{nd}}
=
\frac{\varphi_*\mu_M(\varpi^n\mathcal{O}^m)}{\mu_N(\varpi^n\mathcal{O}^m)}
=
\frac{\varphi_*\left((\mu_M)\vert_{A}\right)(\varpi^n\mathcal{O}^m)}{\mu_N(\varpi^n\mathcal{O}^m)}
=
\frac{\int_{\varpi^n\mathcal{O}^m}\mu_{y}(\varphi^{-1}(y)^{\natural})d\mu_N}{\int_{\varpi^n\mathcal{O}^m}d\mu_N}
\underset{n\rightarrow+\infty}{\longrightarrow}
\mu_0(\varphi^{-1}(0)^{\natural})=\mu(U^{\natural}).
\]Finally, we compute $\mu_{\mathrm{can}}(X^{\natural})$ from $\mu_{i}(U_{i}^{\natural})$
by inclusion-exclusion. Counts of $\mathcal{O}/(\varpi^{n})$-points
on $X$ can also be computed by inclusion-exclusion, as mentioned
above. Adding up the contributions of all $U_{i}$, we obtain:\[
\underset{n\rightarrow +\infty}{\lim}\frac{\sharp X(\mathcal{O}/\varpi^{n})}{q^{nd}}
=
\mu_{\mathrm{can}}(X^{\natural}).
\]\end{proof}

We are now in a position to prove our main technical result for a
large class of $\mathbb{Z}$-schemes of finite type. We show how to
relate the counts $\sharp X(\mathbb{F}_{q}[t]/(t^{n}))$ to the canonical
measure on $X^{\natural}=X^{\mathrm{sm}}(F)\cap X(\mathcal{O})$ when
the residual characteristic is large enough.

\begin{thm} \label{Thm/JetCountCanMeas}

Let $X$ be a $\mathbb{Z}$-scheme of finite type and assume that
$X_{\bar{\mathbb{Q}}}$ is locally complete intersection, of pure
dimension $d$ and has rational singularities. Then for $p$ large
enough, the $\mathcal{O}$-scheme $X_{\mathcal{O}}$ satisfies the
assumptions of Lemma \ref{Lem/JetCountCanMeas}. Moreover, the sequence
$q^{-nd}\cdot\sharp X(\mathbb{F}_{q}[t]/(t^{n})),\ n\geq1$ converges
and its limit is given by:\[
\underset{n\rightarrow +\infty}{\lim}\frac{\sharp X(\mathbb{F}_{q}[t]/(t^{n}))}{q^{nd}}
=
\mu_{\mathrm{can}}(X^{\natural}).
\]\end{thm}

\begin{proof}

Assume that $X_{\mathcal{O}}$ satisfies the hypotheses of Lemma \ref{Lem/JetCountCanMeas}.
From \cite[Prop. 3.0.2.]{AA18}, we obtain that, for $p$ large enough:\[
\underset{n\rightarrow +\infty}{\lim}\frac{\sharp X(\mathbb{F}_{q}[t]/(t^{n}))}{q^{nd}}
=
\underset{n\rightarrow +\infty}{\lim}\frac{\sharp X(\mathcal{O}/(\varpi^{n}))}{q^{nd}}
=
\mu_{\mathrm{can}}(X^{\natural}).
\]Let us now check that for $p$ large enough: (i) $X$ is flat over
$\mathcal{O}$, with l.c.i. geometric fibers of pure dimension $d$
and (ii) $X_{\bar{F}}$ has rational singularities. (ii) follows straightforwardly
from the assumption on $X_{\bar{\mathbb{Q}}}$ by base change, so
let us check (i). By generic flatness, $X_{\mathbb{Z}[\frac{1}{N}]}$
is flat over $\mathbb{Z}[\frac{1}{N}]$ for some $N\geq1$. Thus by
\cite[\href{https://stacks.math.columbia.edu/tag/01UE}{Tag 01UE}, \href{https://stacks.math.columbia.edu/tag/01UF}{Tag 01UF}]{SP}
and \cite[Thm. 13.1.3.]{Gro66}, the locus $U\subseteq X_{\mathbb{Z}[\frac{1}{N}]}$
where the structure morphism $X_{\mathbb{Z}[\frac{1}{N}]}\rightarrow\Spec(\mathbb{Z}[\frac{1}{N}])$
has l.c.i. geometric fibers of pure dimension $d$ is open and contains
$X_{\mathbb{Q}}$, by assumption. By \cite[Cor. 9.5.2.]{Gro66}, we
obtain that $\{s\in\Spec(\mathbb{Z})\ \vert\ U_{s}=X_{s}\}\subseteq\Spec(\mathbb{Z})$
is open i.e. for $N$ large enough, the structure morphism is flat,
with l.c.i. geometric fibers of pure dimension $d$. So the same holds
for $X_{\mathcal{O}}$ for $p\geq N$. \end{proof}

\begin{rmk}

Note that a similar result holds for a scheme $X$ defined over the
ring of integers of some number field, with the same proof.

\end{rmk}

\paragraph*{Application to moduli of totally negative 2-Calabi-Yau categories}

Let us now consider a quotient stack $[X/G]$ coming from a 2-Calabi-Yau
category. If $[X/G]$ is locally modelled on moment maps of totally
negative quivers, then Theorem \ref{Thm/JetCountCanMeas} applies
and we obtain the desired interpretation of limits of jet counts for
a large class of moduli spaces.

\begin{thm} \label{Thm/CanMeas2CYMod}

Suppose that $X$ is of finite type, of pure dimension $d$ and defined
over $\mathbb{Q}$. Assume also that $[X/G]$ satisfies the hypotheses
of Theorem \ref{Thm/RatSgTotNeg2CY}. 

Then for $p$ large enough, the canonical measure $\mu_{\mathrm{can}}$
on $X^{\natural}=X^{\mathrm{sm}}(F)\cap X(\mathcal{O})$ is well-defined.
Moreover, the sequence $q^{-nd}\cdot\sharp X(\mathbb{F}_{q}[t]/(t^{n})),\ n\geq1$
converges and its limit is given by:\[
\underset{n\rightarrow +\infty}{\lim}\frac{\sharp X(\mathbb{F}_{q}[t]/(t^{n}))}{q^{nd}}
=
\mu_{\mathrm{can}}(X^{\natural}).
\]\end{thm}

Note that the assumption that $X$ is of pure dimension $d$ may be
dropped. Since $X$ is locally complete intersection, its connected
components are equidimensional and the counts of jets may be carried
out separately on each connected component with the appropriate value
of $d$.

\begin{rmk} \label{Rmk/CountJetsStacks}

When $G$ is connected, counts of jets on $X$ are equivalent to counts
of jets on the moduli stack $[X/G]$, weighted by their number of
automorphisms:\[
\sum_{z\in [X/G](\mathbb{F}_q[t]/(t^n))/\sim}\frac{1}{\sharp\Aut(z)}=\frac{\sharp X(\mathbb{F}_q[t]/(t^n))}{\sharp G(\mathbb{F}_q[t]/(t^n))}.
\]Indeed, the datum of a principal $G$-bundle $P\rightarrow\Spec(\mathbb{F}_{q}[t]/(t^{n}))$
and a $G$-equivariant morphism $P\rightarrow X$ is equivalent to
the datum of an orbit of $X(\mathbb{F}_{q}[t]/(t^{n}))$ under the
action of $G(\mathbb{F}_{q}[t]/(t^{n}))$. This is due to the fact
that the restricted morphism $P_{\mathbb{F}_{q}}\rightarrow\Spec(\mathbb{F}_{q})$
has a section, by Lang's theorem. By Hensel's lemma, this section
extends to $\Spec(\mathbb{F}_{q}[t]/(t^{n}))$ and so $P\simeq G\times\Spec(\mathbb{F}_{q}[t]/(t^{n}))$.
As a consequence, when $[X/G]$ satisfies the assumptions of Theorem
\ref{Thm/RatSgTotNeg2CY}, the weighted count of jets on $[X/G]$,
once normalized, converges to $\frac{\mu_{\mathrm{can}}(X^{\natural})}{\mu_{\mathrm{can}}(G^{\natural})}$
as $n$ goes to infinity.

\end{rmk}

\bibliographystyle{plain}
\addcontentsline{toc}{section}{\refname}\bibliography{0D__ISTA_Articles_Bibliographie}

\end{document}